\documentclass[a4paper]{article}

\usepackage[utf8]{inputenc}
\usepackage[T1]{fontenc}
\usepackage[english]{babel}
\usepackage{lmodern}
\usepackage{amsmath}
\usepackage{amssymb}
\usepackage{mathrsfs}
\usepackage{amsthm}
\usepackage{color}
\usepackage{dsfont}
\usepackage{soul}
\usepackage{diagbox}
\usepackage{ulem}
\usepackage{verbatim}
\usepackage{moreverb}
\usepackage{url}
\usepackage{graphicx}
\usepackage{subfig}
\usepackage{listings}
\usepackage[breaklinks]{hyperref}
\usepackage{hyperref}
\usepackage[top=2.5cm, bottom=2cm, left=2.5cm, right=2cm]{geometry}
\hypersetup{urlcolor=blue,linkcolor=blue,citecolor=blue,colorlinks=true}

\newtheorem{theo}{Theorem}[section]

\newtheorem{lem}{Lemma}[section]
\newtheorem{prop}{Proposition}[subsection]

\newtheorem{Cor}{Corollary}[section]

\newcommand{\R}{\mathbb{R}}

\newcommand{\N}{\mathbb{N}}
\newcommand{\E}{\mathbb{E}}

\DeclareMathOperator{\argmin}{argmin}
\DeclareMathOperator{\Id}{Id}
\DeclareMathOperator{\Proj}{Proj}
\DeclareMathOperator{\tr}{tr}
\DeclareMathOperator{\pen}{pen}
\DeclareMathOperator{\Var}{Var}

\newcommand{\email}[1]{\href{mailto:#1}{\textsf{#1}}}

\title{Gaussian linear model selection in a dependent context}    
   
\author{Emmanuel Caron\footnote{Emmanuel Caron, Université Lumière Lyon 2, Laboratoire ERIC EA3083, 69007 Lyon, France. \newline Email: \email{emmanuel.caron@univ-lyon2.fr}} \and J\'er\^ome Dedecker\footnote{Jérôme Dedecker, Université de Paris, Laboratoire MAP5 UMR 8145, 75006 Paris, France. \newline Email: \email{jerome.dedecker@parisdescartes.fr}} \and Bertrand Michel\footnote{Bertrand Michel, Ecole Centrale de Nantes, Laboratoire de Mathématiques Jean Leray UMR 6629, 44300 Nantes, France. \newline Email: \email{bertrand.michel@ec-nantes.fr}}}     
        
\begin{document}

\maketitle

\pagestyle{plain}

\begin{abstract}
In this paper, we study the nonparametric  linear model, when the error process is a dependent Gaussian process. We  focus on the estimation of the mean vector via a model selection approach. We first give the general theoretical form of the penalty function, ensuring that the penalized estimator among a collection of models satisfies an oracle inequality. Then we derive a penalty shape involving the spectral radius of the covariance matrix of the errors, which can be chosen proportional to the dimension when the error process is stationary and short range dependent. However, this penalty can be too rough in some cases, in particular when the error process is long range dependent. In a second part, we focus on the fixed-design regression model assuming that the error process is a stationary Gaussian process. We propose a model selection procedure in order to estimate the mean function via piecewise polynomials on a regular partition, when the error process is either short range dependent, long range dependent or  anti-persistent. We present different kinds of penalties, depending on the memory of the process. For each case, an adaptive estimator is built, and the rates of convergence are computed.  Thanks to several sets of simulations, we study the performance of these different penalties for all types of errors (short memory, long memory and anti-persistent errors). Finally, we give an application of our method to the well-known Nile data, which clearly shows that the type of dependence of the error process must be taken into account.
\end{abstract}



\section{Introduction}


Let us consider the  linear model
\begin{equation}\label{genericModel}
Y = t^{\ast} + \varepsilon,
\end{equation}
where $Y$ is the $n$-dimensional vector of observations, $t^*$ is an unknown (deterministic) vector to be estimated, and $\varepsilon$ is the vector of errors. It is well know that Model \eqref{genericModel} can serve as a canonical model to express a large class of statistical problems (see~\cite{birge2001gaussian}). 
In this paper, we focus on the estimation of the  vector $t^{\ast}$ with a model selection approach, in the general framework where the error process $\varepsilon$ is a dependent Gaussian random vector,  with  covariance matrix  $\Sigma$.
Our first goal is to give the theoretical form of the penalty function, depending on $\Sigma$, ensuring that the penalized estimator  among a collection of models satisfies an oracle inequality. 

This model has been widely studied for independent and identically distributed (i.i.d.) errors, in particular by Birg\'e and Massart in the Gaussian case~\cite{birge2001gaussian}. Baraud worked in the general i.i.d. case with a deterministic design first~\cite{baraud2000model}, then with a random design~\cite{baraud2002model}. Some extensions of these results to a $\beta$-mixing framework are presented in~\cite{baraud2001adaptive}.
 The idea of using a penalty function goes back to the pioneering works of Akaike~\cite{akaike1973information} and Mallows~\cite{mallows1973some}. Later, Birg\'e and Massart developed a non-asymptotic approach to the selection of penalized models~\cite{birge2001gaussian}, \cite{birge2001generalized}, \cite{birge2007minimal}.

We follow in this paper the strategy developed by Birg\'e and Massart which is based on a non-asymptotic control of the fluctuations of the empirical contrast.

Let us be more precise here. In order to find a linear subspace that realizes a bias-variance tradeoff, let us introduce a finite collection of models $\{S_{m}, m \in \mathcal{M}\}$, denoting by $d_{m}$ the dimension of $S_{m}$. Let  then $\hat{t}_{m}$  be the least squares estimator $\Proj_{S_{m}} (Y)$ of $t^{\ast}$ on $S_{m}$.
A penalization strategy is used by selecting a model with a criterion of the form
\[\hat{m} \in \mathrm{argmin}_{m \in \mathcal{M}} \left\{ \left \| Y - \hat{t}_{m} \right \|_{n}^{2} + \pen(m) \right\},\]
where $\| \cdot \|_n$ denotes the  (normalized) euclidean norm in ${\mathbb R}^n$, and  $\pen : \mathcal{M} \rightarrow \mathbb{R}^{+}$ is a penalty function defined on the family of models. 
Following the Birg\'e and Massart approach, we derive a penalty function which provides an oracle inequality for the model selection procedure in the dependent Gaussian framework.

In Section~\ref{sec::general_setting}, a general penalty shape is presented. The main term is the quantity $\tr (\Proj_{S_{m}} \Sigma)$ (tr denoting the trace) which plays the same role as the  term  $\text{Var}(\varepsilon_1) \, d_{m}$ in the results of Birg\'e and Massart for i.i.d. Gaussian errors. Similar penalties have already been introduced by Gendre~\cite{gendre2014model} in the context of model selection for additive regression. 
However Gendre \cite{gendre2014model} is not interested in the same questions as us: he is concerned with additive regression whereas our objective is to study the Gaussian regression with dependent errors. In the same way as for us, the analysis of \cite{gendre2014model} is based on a general Gaussian model selection, but it appears that for our concern, the general penalty form we provide is more appropriate than that provided by \cite{gendre2014model}.
In addition, the assumptions of \cite{gendre2014model} do not apply to the context of long range dependent or anti-persistent errors. 

Note that the trace  $\tr \left( \Proj_{ S_{m}} \Sigma \right)$ is bounded by $d_{m} \rho(\Sigma)$, where $\rho (\Sigma)$ is the spectral radius of the covariance matrix. Hence, neglecting some residuals terms (see Section~\ref{sec::general_setting}),  the  following penalty can be used: for any $K>1$,
\begin{equation}\label{simplepen}
\pen(m) \geq   K  \frac{\rho(\Sigma) d_m}n \, .
\end{equation}
For instance, if we suppose that the error process is a short memory stationary process with bounded spectral density, then the spectral radius is bounded, and this penalty shape is very closed to the i.i.d. case up to a constant. The penalty can still be chosen proportional to the dimension, as in the i.i.d. case, but the usual variance term is replaced by the spectral radius of the covariance matrix. 

However, the penalty \eqref{simplepen} may be too rough in some cases, in particular if the error process is long range dependent.  To see how to handle this case in a concrete situation, we study in Sections~\ref{sec::short_long_mem} and~\ref{sec::simus} the fixed-design regression model
\begin{equation}\label{fixedreg}
Y_i= f^* \left ( \frac i n \right ) + \varepsilon_i \, , 
\end{equation}
where $(\varepsilon_i)_{i \geq 1}$ is a stationary Gaussian process. By standard arguments,  this model can be written as a special case of the generic Model \eqref{genericModel}. 

Note that Model \eqref{fixedreg} has been widely studied in the literature (with possibly non Gaussian errors) via kernel or wavelets methods. 

For kernel estimators, let us first quote the paper by Hall and Hart \cite{HallHart1990Long}, who considered a particular class of Gaussian errors. The authors showed in particular that, for a twice differentiable function $f^*$, the rate is the same as in the i.i.d. case if and only if $ \sum_{k>0} |\mathrm{Cov}(\varepsilon_1, \varepsilon_k)| < \infty$, and they gave minimax rates in the long range dependent case. 
 Let us also cite the papers by Cs\"{o}rg\H{o} and Mielniczuk \cite{CsMi1995Short}, \cite{CsMi1995Long}, 
 \cite{CsMi1995Long2} (long memory is considered in \cite{CsMi1995Long} and \cite{CsMi1995Long2}), Tran et al \cite{RoussasEtAl1996Short} (short memory case), and Robinson \cite{Rob1997Dep}. Robinson's article provides very general results for  short range and long range dependent processes, and rates of convergence for anti-persistent errors (also called negatively correlated errors) can be derived from his Lemma 3. Local polynomial fitting with long memory, short memory and anti-persistent errors is considered by Beran and Feng~\cite{BeFe2002Long}. Note that none of these articles adresses the issues of adaptive estimation or data-driven bandwidth selection. 
 
For wavelets type estimators,  let us first quote the paper by Wang \cite{Wang1996LongMemo}, who gave minimax results in the long range dependent case, when the function $f^*$ belongs to a Besov class. Let us also cite the papers by Johnstone and Silverman \cite{JoSi1997Corr}, Johnstone \cite{Jo1999Corr},  and more recently Li and Xiao \cite{LiXi2007Long}  and Beran and Shumeyko \cite{BeSh2012Long}. These four papers addressed the issue of a data-driven choice of the threshold. Theorem 1 in \cite{Jo1999Corr} gave a very precise minimax result (up to  constants), but for an  asymptotic model which is a bit different from \eqref{fixedreg} (see the discussion at the end of the paper \cite{Jo1999Corr}).  By adapting the block thresholding method described in Hall et al \cite{HKP1999} to the long memory case, Li and Xiao \cite{LiXi2007Long} showed that the block thresholded wavelets estimators are adaptive and minimax for a large class of functions.

In Sections~\ref{sec::short_long_mem} and~\ref{sec::simus} of the  present paper, we propose a model selection procedure to estimate $f^*$ via piecewise polynomials on a regular partition of size $m$. The choice of piecewise polynomials is very natural here, since the function $f^*$ is supported on $[0,1]$, and such estimators do not show  bad behaviors near the boundary. We show that
 
\begin{itemize}
\item For short memory error processes (i.e. when $\rho(\Sigma)$ is uniformly bounded) the penalty is of the form 
$$\text{pen}(m)= K \frac m n   $$ (for some constant $K>0$ to be calibrated),  the penalized estimator is adaptive with respect to the unknown regularity of the function $f^*$,  and yields the same rates of convergence as in the i.i.d setting. 
\item For long memory processes, that is when the auto-covariances $\gamma_\varepsilon(k)$ of the error process are such that
\[ | \gamma_\varepsilon (k) | \leq \kappa k^{-\gamma}, \qquad \text{for some} \ \kappa > 0 \ \text{and} \ \gamma \in (0,1),\]
the penalty is a concave function of $(m/n)$
$$\text{pen}(m)= K  \left( \frac m n \right )^\gamma $$ (for some constant $K>0$ to be calibrated), the penalized estimator is adaptive with respect to the unknown regularity of the function $f^*$,  and yields the same minimax rates of convergence as in \cite{Wang1996LongMemo}. 
 
\item For anti-persistent errors such that 
$$
\mathrm{Var}(\varepsilon_1 + \cdots  + \varepsilon_n)  \leq \kappa n^{2-\gamma}, \qquad \text{for some} \ \kappa > 0 \ \text{and} \ \gamma \in (1,2) ,
$$
and in the case of regressorams (piecewise polynomials of degree 0), the penalty has the form 
$$\text{pen}(m)= K  \left( \frac{m^\gamma}{n^\gamma}  + \frac{\log (m)}{n} \right ) $$  (for some constant $K>0$ to be calibrated). The main part of the penalty is then a convex function of $(m/n)$. The 
penalized estimator is adaptive with respect to the unknown regularity of the function $f^*$,  and yields faster rates of convergence than in the i.i.d setting.  Note that similar rates can also be deduced from Lemma 3 in \cite{Rob1997Dep}. 
\end{itemize} 
 
In Section~\ref{sec::simus}, we simulate different kind of short memory processes (a Gaussian ARMA(2,1) process, two non Gaussian $\beta$-mixing Markov chains), of long memory processes (a fractional Gaussian noise with Hurst index in (1/2,1), and a non Gaussian $\beta$-mixing Markov chain), and an anti-persistent process (a fractional Gaussian noise with Hurst index in (0, 1/2)). For regressograms on a regular partition of size $m$, we investigate different kind of penalties:  the usual penalty proportional to $m/n$, a penalty proportional to $(m/n)^\gamma$  in the case of long range dependent or anti-persistent errors, and some penalties for which $\gamma$ is estimated via an estimator of the Hurst index based on the $Y_i$'s or on the residuals.  Finally, an important message of this paper is that the slope heuristics~\cite{birge2007minimal} can be adapted to calibrate penalties in the context of regression with dependent errors. 

In Section~\ref{sec::data_Nil}, we give an application of our method to the well known Nile data, and we continue the discussion started in Robinson's article   \cite{Rob1997Dep}. In Section~\ref{sec::discuss}, we discuss other possible applications of the generals results of Section~\ref{sec::general_setting}. Finally, Section~\ref{sec::proofs} is devoted to the proofs of the results of Sections~\ref{sec::general_setting} and~\ref{sec::short_long_mem}.
 



\section{A Gaussian linear model selection theorem in a dependent context}
\label{sec::general_setting}
 
\subsection{General setting} 

Recall the equation of the Gaussian linear model~\eqref{genericModel}
$$
Y  = t^{\ast}  + \varepsilon,
$$
where the mean vector $t^{\ast}$ belongs to  $\R^n$ and where the error vector $\varepsilon$ is a Gaussian random vector. We consider the general setting where the components of $Y$ are not independent  
\begin{equation*}
\label{eq:}
\varepsilon \sim \mathcal{N}_n(0, \Sigma).
\end{equation*}
The covariance matrix $\Sigma$ is a $n \times n$ semidefinite matrix with eigenvalues $\lambda_1 \geq \dots \geq\lambda_ n\geq0$. We also introduce the spectral radius of $\Sigma$
\[\rho(\Sigma) = \max_{1 \leq i \leq n} \lambda_{i} = \lambda_1 .\] 

The aim is to estimate the unknown vector $t^{\ast} $ from the observation $Y$. One standard strategy is to constrain the estimator to belong to a given linear subspace $S $ of $\R^n$. Let $\| \cdot \|_n$ denotes the (normalized) euclidean norm in $\R ^n$
\[ \| t \|_n ^2 =  \frac 1n \sum_{i=1}^n  t_i^2 .\]
The least squares contrast  is defined for $t \in \R^n$ by
\[\gamma_n(t) = \left \| Y - t \right \|_n^{2},\]
and the  minimizer of $\gamma$  over $S$ is the orthogonal projection of $Y$ on $S$
\[ \Proj_{S}(Y) = \argmin_{t \in S} \gamma_n(t).\]
With a slight abuse of notation, we shall write $\Proj_{S}$ for the projection operator on  $S$ and for its matrix on the canonical basis.
The  $\ell^{2}$ risk of an estimator $\hat t$ is defined by 
\[R(\hat t) = \mathbb{E} \left[ \left \| \hat t   - t^{\ast} \right \|_n^{2} \right] , \]
where the expectation is under the distribution of $Y$.  Using Pythagoras equality  in $\R^n$ together with \eqref{genericModel}, we find that the risk of $\Proj_{S}(Y) $ satisfies the following bias-variance decomposition 
\[\mathbb{E} \left[ \left \| t^{\ast} - \Proj_{S}(Y) \right \|_n^{2} \right] = \left \| (\Id - \Proj_{S}) t^{\ast} \right \|_n^{2} + \mathbb{E} \left[ \left \|  \Proj_{S} (\varepsilon) \right \|_n^{2} \right].\]
The bias $\left \| (\Id - \Proj_{S}) t^{\ast} \right \|_n^{2}$ is small for large enough linear subspace $S$. It can be easily checked that the variance term is equal to 
$\ \mathbb{E} \left[ \left \|  \Proj_{S} (\varepsilon) \right \|_n^{2} \right] = \frac{1}{n} \tr \left(   \Proj_ S  \Sigma  \right)$, see the proof of Theorem~\ref{theo:main_theorem}.  As the i.i.d. case, the variance term tends to increase with the dimension of $S$.

In order to find a linear subspace that realizes a bias-variance tradeoff, we introduce a finite collection of linear subspaces $\{S_{m}, m \in \mathcal{M}\}$ that we call {\it models}, and we denote by $d_{m}$ the dimension of $S_{m}$. For $m \in \mathcal{M}$, we denote by $ \hat{t}_{m}$ the least squares estimator $\Proj_{S_m}(Y) $ of $t^{\ast}$ on $S_m$. 
We also introduce the oracle model $m_{0}$, that is the model that provides the least squares estimator with minimum risk
\[m_{0} \in \mathrm{argmin}_{m \in \mathcal{M}} \{R(\hat{t}_{m})\}.\]
Now the aim is to select a model in the collection such that the  risk of the selected estimator is as close as possible to the oracle model.

The true risk $  R(\hat{t}_{m}) $ of $\hat{t}_{m}$  being unknown in practice, we introduce the empirical risk
\[\widehat{R}(\hat{t}_{m}) =  \left \| Y - \hat{t}_{m} \right \|_n^{2}.\]
Obviously this criterion can not be used to select a model in the collection because of the overfitting effect. We follow a penalization strategy ~\cite{akaike1973information,mallows1973some,birge2001gaussian,massart2007concentration} by selecting a model with a criterion of the form
 \begin{equation}
\hat{m} \in \mathrm{argmin}_{m \in \mathcal{M}} \left\{ \left \| Y - \hat{t}_{m} \right \|_n^{2} + \pen(m) \right\},
\label{crit}
\end{equation}
where  $\pen : \mathcal{M} \rightarrow \mathbb{R}^{+}$ is a penalty function defined on the family of models.  In this paper we perform a non asymptotic analysis of the risk of the selected estimator $\hat{t}_{\hat m}$ \cite{massart2007concentration}. By this way we derive a penalty function which provides an oracle inequality for the model selection procedure, in the dependent Gaussian context.

\subsection{A general Gaussian model selection result}

Let $\pi = \{\pi_m , \, m \in \mathcal M  \} $ be a probability measure defined on $\mathcal M$ : $\sum_{m \in \mathcal M} \pi_{m} = 1$. We first give a general shape for the penalty function and the corresponding oracle inequality.
\begin{theo}
\label{theo:main_theorem}
For some constant $K > 1$, for any penalty function $\pen : \mathcal M \rightarrow \R^+$  such that for any $m \in \mathcal M$,
\begin{equation}
\pen(m)  \geq   \frac Kn  \left(  \sqrt{ \tr \left( \Proj_{ S_{m}} \Sigma   \right) +  \rho(\Sigma) }+ \sqrt{\rho(\Sigma)} \sqrt{2 \log \left( \frac{1}{\pi_{m}} \right) }  \right)^{2},
\label{eq:pen0}
\end{equation}
then there exists a constant $C> 1$ which only depends on $K$   such that the estimator $\hat{t}_{\hat m}$ selected by the criterion \eqref{crit} satisfies 
\begin{equation}
\mathbb{E} \left[ \left \| t^{\ast} - \hat{t}_{\hat m} \right \|_n^{2} \right] \leq C \left( \inf_{m \in \mathcal M} \left\{ \mathbb{E} \left[ \left \| t^{\ast} - \hat{t}_{m} \right \|_n^{2} \right] + \pen(m) \right\} + \frac {\rho(\Sigma)}{n} \right) .
\label{eq:riskbound}
\end{equation}
\end{theo}

The main term in the penalty shape \eqref{eq:pen0} is the trace term $ \tr \left( \Proj_{ S_{m}} \Sigma   \right) $.  This quantity plays the same role as the term $\Var(\epsilon_{1}) d_m$ in the results of Birg\'e and Massart \ for independent Gaussian errors~\cite{birge2001gaussian,massart2007concentration}. 
Of course, this penalty can only be calculated if the  matrix $\Sigma$ is completely known. However we will see that, in certain cases, we can consider effective strategies to circumvent this issue (see Sections~\ref{sec::short_long_mem} and~\ref{sec::simus}).

We can propose penalty shapes from the upper bounds
\begin{equation*}
\tr \left( \Proj_{ S_{m}} \Sigma   \right)  \leq      \sum_{i=1}^{d_m} \lambda_{i} 
 \leq    d_{m}  \rho(\Sigma).
\label{chap4:somme_lambda}
\end{equation*}
Actually, with a minor modification of the proof of Theorem~\ref{theo:main_theorem}, it can be checked that the risk bound \eqref{eq:riskbound} is still valid when replacing the lower bound in \eqref{eq:pen0}  by
\begin{equation*}
\pen(m) \geq   \frac Kn \left( \sqrt{\sum_{i=1}^{d_{m}} \lambda_{i}} + \sqrt{\rho(\Sigma)} \sqrt{2 \log \left( \frac{1}{\pi_{m}} \right)}  \right)^{2},
\end{equation*}
 or by 
 \begin{equation}
\pen(m) \geq   K  \frac{\rho(\Sigma)}n  \left( \sqrt{ d_{m}}  +  \sqrt{2 \log \left( \frac{1}{\pi_{m}} \right)}   \right)^{2},
\label{eq:pen2}
 \end{equation}
for any $K>1$.

If the sequence  $(\varepsilon_{i})_{i \geq 1}$ is a stationary and short memory Gaussian process, then the spectral radius is bounded and the penalty shape~\eqref{eq:pen2} is completely in line with the case of  independent Gaussian errors~\cite{birge2001gaussian,massart2007concentration}, the usual variance term $\Var(\epsilon_{1})$ being replaced by the spectral radius $\rho(\Sigma)$. 

The three penalty shapes given below depend on the probability  $\pi$. If the collection of model is not too rich (see for instance \cite{birge2001gaussian,massart2007concentration} or Chapter 2 in~\cite{giraud2014introduction}), it might be chosen in such a way that 
$$
\rho(\Sigma)  \log \left( \frac{1}{\pi_{m}} \right)
$$
is smaller or of the same order as the main terms $\tr \left( \Proj_{ S_{m}} \Sigma   \right)$, $\sum_{i=1}^{d_{m}} \lambda_{i} $ or  $d_{m}\rho(\Sigma)$. 
To sum up, if the spectral radius is bounded  and if the collection of models is not too rich we see that the penalty can be chosen proportional to the dimension $d_m$, as in the independent case.

It is tempting to keep the penalty shape \eqref{eq:pen2} as a general penalty shape for Gaussian linear model selection with dependent errors. However, as we will see later in the paper, this penalty shape is too rough in some cases. For instance, it cannot lead to  minimax rates of convergence for non parametric regression with long range dependent errors (see Subsection~\ref{sub:case_long_range}). 
 
At this point, it should be clearly quoted that a penalty similar to \eqref{eq:pen0} has been given in the paper~\cite{gendre2014model}. The main difference is that, in the inequality similar to~\eqref{eq:riskbound} proved in~\cite{gendre2014model} (Inequality (2.2) of Theorem 2.1 in~\cite{gendre2014model}), the residual term is $\rho(\Sigma) R_n/n$ instead of 
$\rho(\Sigma)/n$. For the questions he has in mind (which are not directly related to  time series), Gendre is able to effectively control this additional term $R_n$. But it does not seem easy to handle for long range dependent errors or anti-persistent errors, which are precisely the kind of error processes that we want to study in the present paper.

\section{Non parametric regression with Gaussian dependent errors}
\label{sec::short_long_mem}

In this section we study the fixed design regression problem with dependent Gaussian errors.
Let $f^{\ast}$ be a function in $\mathbb L^\infty([0,1]) $, and recall the equation of model~\eqref{fixedreg}
$$
Y_i = f^{\ast}\left(\frac i n\right) + \varepsilon_i , \quad i \in  \{1, \dots n \},
$$
where $( \varepsilon_1, \dots, \varepsilon_n) \sim \mathcal N _n (0,\Sigma_n) $. 
The aim is to estimate  $f^{\ast}$  thanks to the observations $Y_1, \dots, Y_n$.	

By considering the application 
\[ f \in \mathbb L^{\infty}([0,1])     \mapsto I(f) =  \left( f(1/n), \dots,  f(1) \right)  \in  \R^n,\] 
we can easily associate a linear subspace of $\R^n$ to any linear subspace of $\mathbb L^\infty ([0,1])$.
Slightly abusing the notation, we identify the function $f$ to the vector $I(f)$, and we write
\[\| f\|_n^{2}  = \frac 1n \sum_{i=1}^n f ^2(i/n) = \|  I(f)\|_n^{2}   , \: \textrm{ for  } f \in \mathbb L^\infty ([0,1]) .\]
For $F$  a finite linear subspace of $\mathbb L^\infty ([0,1]) $, we define the least-squares estimator  $\hat{f}$ of $f^\ast$ on  $F$ as
\[ \hat{f}  =   \argmin_{f \in F }  \|  Y - f  \|_n^2\, , 
\quad \text{where} \quad   \| Y - f \|_n^2 = \frac 1n \sum_{i=1}^n (f  (i/n)  - Y_i)^2.
 \]

We shall only consider here the linear spaces $S_m$  of ${\mathbb R}^n$  induced by the linear space
$F_m$ of  $\mathbb L^\infty ([0,1]) $ generated by the family of piecewise polynomials of degree at most $r$ ($r \in {\mathbb N}$)  on the regular partition of size $m$ of the interval $[0,1]$. Obviously, the linear space $S_m$ has dimension $d_m=(r+1)m$; the case $r=0$ corresponds to the regular regressogram of size $m$. 

We denote by $\hat f_m$ the least square estimator of $f^*$ on $F_m$. 

We shall always consider some weights $\pi_m$ of order $m^{-2}$ (suitably normalized in such a way that 
$\sum_{m=1}^n \pi_m=1$). For such weights, the terms involving $\pi_m$ in the general penalty \eqref{eq:pen0} is of order $\rho(\Sigma) \log (m)$;  in the applications given below, it will be negligible with respect to the main term $\tr \left( \Proj_{ S_{m}} \Sigma   \right)$. 

\subsection{The case of short range dependent sequences}
\label{SRD}

In this subsection, we assume that the error process $(\varepsilon_i)_{i \geq 1}$ is  stationary and short-range dependent. By short range dependent, we mean  that 
\begin{equation}\label{boundedradius}
\rho_\varepsilon = \sup_{n \in \N^ *} \rho (\Sigma_n)  < \infty .
\end{equation}
Note that \eqref{boundedradius} is satisfied as soon as the spectral density of $(\varepsilon_i)_{i \geq 1}$ is bounded, which corresponds to the usual definition of short range dependency. 

In this  setting, the model selection procedure is exactly the same as in the i.i.d. framework, by replacing the variance of the errors by the spectral radius in  the penalty. More precisely, we obtain a penalty of the form
\begin{equation*} 
 \pen(m) =K \rho_\varepsilon \frac m  n\, ,
 \end{equation*}
  for some positive constant $K$ depending on the the degree $r$. 
  We now select a model in  $\mathcal M _n$ according to the criterion \eqref{crit}, which can be rewritten as
 \begin{equation}
\hat{m} \in \mathrm{argmin}_{m \in \{1, \ldots, n\}} \left\{ \left \| Y - \hat{f}_{m} \right \|_n^{2} + \pen(m) \right\}.
\label{critnonparam}
\end{equation}
  
Following  \cite{baraud2000model}, we  derive rates of convergence when $f^*$ belongs to some  Besov spaces $\mathcal B_{\alpha,\ell,\infty}$ for $\ell^{-1} <\alpha < r+1 $ and $\ell \geq 2$ (see \cite{devore1993constructive} for the definition of Besov spaces). 
In short,  the approximation term in the risk decomposition of $ \hat f_m$ satisfies (see Sections 4 and 7.4 in \cite{baraud2000model})
\begin{equation}\label{bias}
 \inf_{g \in F_m}  \left\| f^{\ast} -   g \right\|_n^{2}  \leq  C(\alpha, r)  |f^{\ast}|^2_{\alpha,\ell} \left(m^{-2\alpha} + n^{-2\alpha +2/\ell} \right)\, ,
\end{equation}
where $|\cdot |_{\alpha, \ell}$ is the usual norm on $\mathcal B_{\alpha,\ell,\infty}$. 
Balancing the variance term and the approximation terms exactly as in case of i.i.d errors, we end up with the same rate of convergence as in the i.i.d. case

\begin{Cor}
Let $(\ell, \alpha)$ be such that $\alpha \in (0, r+1) $ and $ \ell\geq \max (2, (2\alpha+1)/(2 \alpha^2))$. 
For a stationary Gaussian process satisfying \eqref{boundedradius}, and for the estimator $\hat f _{\hat m}$ selected according to the penalized criterion procedure~\eqref{critnonparam}, 
$$\sup_{|f^\ast|_{\alpha, \ell } \leq L} \mathbb E \left \|f^\ast - \hat f _{\hat m} \right  \|^2_n  \leq C n^{-\frac{2\alpha}{2 \alpha +1}},$$
where  $C$ depends on $\rho_\varepsilon$, $K$, $\alpha$, $\ell$ and $L$.
\end{Cor}
This upper bound is known to be the minimax rate of convergence for the estimation of $f^\ast$ in the i.i.d. case. This is satisfactory since a sequence of i.i.d. Gaussian random variables  is of course  short-range dependent. 

As for the Gaussian i.i.d case, the penalty  is defined up to a multiplicative constant $K$. The spectral radius is unknown, as is the variance of the errors in the standard i.i.d. setting. In practice, the penalty is chosen proportional to the model dimension $m$ and calibrated according to the slope heuristic method introduced by
Birg\'e et Massart~\cite{birge2001generalized}, see Section~\ref{subs:slopeh} further.
 
\subsection{The case of long range dependent sequences}
\label{sub:case_long_range}
 
In this subsection, we assume that the error process $(\varepsilon_i)_{i \geq 1}$ is strictly stationary, but we do not assume that  \eqref{boundedradius} holds. Instead, we assume that 
\begin{equation}\label{lrd}
    |\gamma_\varepsilon (k)| \leq \kappa k^{-\gamma},  \quad \text{for some $\kappa >0$ and $\gamma \in (0,1)$,}
\end{equation}
where $\gamma_\varepsilon(k)$ is the auto-covariance $\gamma_\varepsilon(k)= \mathrm{Cov} ( \varepsilon_0, \varepsilon_k)$.
Of course, \eqref{lrd} is only an upper bound, so that it may happen that 
$ \sum_{k>0} |\gamma_\varepsilon(k)| < \infty$; in such a case \eqref{boundedradius} holds and the process in short range dependent. But the interesting case is of course when  $|\gamma_\varepsilon (k)|$ is exactly of order  $k^{-\gamma}$, so that $ \sum_{k>0} |\gamma_\varepsilon(k)| = \infty$. This is what we mean here by long range dependent.
 
To control the main term of the penalty, we shall prove the following lemma

\begin{lem}\label{mainLem}
Let $S_m$ be the linear space  of ${\mathbb R}^n$  induced by the family of piecewise polynomials of degree at most $r$  on the regular partition of size $m$ of the interval $[0,1]$. If \eqref{lrd} holds, then 
$$
 \tr \left( \Proj_{ S_{m}} \Sigma   \right) \leq C m^\gamma n^{1-\gamma}\, ,
$$
where $C$ depends on $\kappa, \gamma$ and $r$. 
\end{lem}

Moreover, by the classical Gerschgorin theorem, we easily see that
$$
\rho(\Sigma_n)  \leq B n^{1-\gamma} \, ,
$$
where $B$ depends on $\kappa$ and $\gamma$. Combining this last bound with Lemma \ref{mainLem}, we infer from  \eqref{eq:pen0} that one can choose a penalty of the form
\begin{equation*} 
 \pen(m) =K  \frac{m^\gamma}{n^\gamma} \, ,
\end{equation*}
for some positive constant $K$ depending on $\kappa, \gamma$ and $r$.
 
Now, since the bias term \eqref{bias} is still valid for any function $f^*$ in the Besov space $\mathcal B_{\alpha,\ell,\infty}$ (with $\ell^{-1} <\alpha < r+1 $ and $\ell \geq 2$), we can proceed as in Section \ref{SRD} to get the rate of convergence of the estimator $\hat f _{\hat m}$. The difference is that the bias-variance problem consists of  balancing two terms of order 
$$
     \frac{1}{m^{2\alpha}} \ \text{(bias)} \quad  \text{and} \quad  \frac{m^\gamma}{n^\gamma}\  \text{(variance).}
$$
This leads to the following corollary
 
\begin{Cor}
Let $(\ell, \alpha)$ be such that $\alpha  \in (0, r+1)$ and $ \ell\geq \max(2, (2\alpha+\gamma)/(2 \alpha^2))$. 
For a stationary Gaussian process satisfying \eqref{lrd}, and for the estimator $\hat f _{\hat m}$ selected according to the penalized criterion procedure~\eqref{critnonparam}, 
$$\sup_{|f^\ast|_{\alpha, \ell } \leq L} \mathbb{E} \left \|f^\ast - \hat f _{\hat m}  \right \|^2_n  \leq C n^{-\frac{2\alpha \gamma}{2 \alpha +\gamma}},$$
where  $C$ depends on  $\gamma$, $K$, $\alpha$, $\ell$ and $L$.
\end{Cor}

This rate is satisfactory, since it corresponds to the minimax rates described in the same setting by Wang \cite{Wang1996LongMemo} when $\gamma_\varepsilon(k)$ is exactly of order $k^{-\gamma}$. Note however that the minimax rate in \cite{Wang1996LongMemo} is written for the usual ${\mathbb L}^2([0,1])$-norm.

Let us make some additional comments: if the exponent $\gamma$ is known, then the slope heuristic can still be used to calibrate the other constants in the penalty term. We shall see that it works pretty well in the simulation section and we will also investigate the calibration of $\gamma$ for the more general and difficult framework where the exponent $\gamma$ is unknown.

\subsection{Regular regressograms and anti-persistent errors}
  \label{sub:case_very_short_range}

We now assume that the sequence $(\varepsilon_i)_{i \geq 1}$ is stationary and anti-persistent in the following sense: there exist a parameter $\gamma \in (1,2)$ and a positive constant $\kappa$ such that Condition~\eqref{boundedradius} holds and
\begin{equation}\label{vsd} \quad \text{Var} \left (  \sum_{k=1}^n \varepsilon_k \right ) \leq \kappa n^{2- \gamma} \, .
\end{equation}
For instance, Conditions~\eqref{boundedradius}  and~\eqref{vsd} hold if $(\varepsilon_i)_{i \geq 1}$ is a fractional Gaussian noise with Hurst index $H \in (0, 1/2)$ (see Section~\ref{subs:slopeh} for the definition of the Hurst index). In that case, $\gamma=2-2H$. The term anti-persistent is borrowed from this particular case. 

In this subsection, we only consider the case of regular regressograms, which corresponds to  estimators  via piecewise polynomials of degree 0 on a regular partition of $[0,1]$. 

To control the main term of the penalty, we shall prove the following lemma

\begin{lem}\label{newLem}
Let $S_m$ be the linear space  of ${\mathbb R}^n$  induced by the family of indicators of intervals  on the regular partition of size $m$ of the interval $[0,1]$. If Conditions~\eqref{boundedradius}  and~\eqref{vsd} hold, then 
$$
 \tr \left( \Proj_{ S_{m}} \Sigma   \right) \leq C m^\gamma n^{1-\gamma}\, ,
$$
where $C$ depends on $\kappa$ and $\gamma$. 
\end{lem}

We infer from  \eqref{eq:pen0} that one can choose a penalty of the form
\begin{equation*} 
 \pen(m) =K  \left (\frac{m^\gamma}{n^\gamma}  + \frac{\log (m)}{n} \right )\, ,
\end{equation*}
for some positive constant $K$ depending on $\kappa, \gamma$ and $\rho_\varepsilon$ (recall that $\rho_\varepsilon$ is the constant appearing in \eqref{boundedradius}). 
 
Now, since the bias term \eqref{bias} is still valid for any function $f^*$ in the Besov space $\mathcal B_{\alpha,\ell,\infty}$ (with $\ell^{-1} <\alpha < 1 $ and $\ell \geq 2$), we can proceed as in Section~\ref{SRD} to get the rate of convergence of the estimator $\hat f _{\hat m}$.
This leads to the following corollary
 
\begin{Cor}\label{vsdCor}
Let $(\ell, \alpha)$ be such that $\alpha  \in (0, 1)$ and $ \ell\geq \max(2, (2\alpha+\gamma)/(2 \alpha^2))$. 
For a stationary Gaussian process satisfying Conditions~\eqref{boundedradius}  and~\eqref{vsd}, and for the estimator $\hat f _{\hat m}$ selected according to the penalized criterion procedure~\eqref{critnonparam}, 
$$\sup_{|f^\ast|_{\alpha, \ell } \leq L} \mathbb{E} \left \|f^\ast - \hat f _{\hat m}  \right \|^2_n  \leq C n^{-\frac{2\alpha \gamma}{2 \alpha +\gamma}},$$
where  $C$ depends on  $\gamma$, $K$, $\alpha$, $\ell$ and $L$.
\end{Cor}

It is interesting to notice that, for a regularity $\alpha <1$, the rate of convergence given in Corollary 
\ref{vsdCor} is faster than in the case where the sequence $(\varepsilon_i)_{i \geq 1}$ is i.i.d.


\section{Numeric experiments}
\label{sec::simus}

\subsection{Slope heuristics} \label{subs:slopeh}

For the results given in the previous sections, the penalty functions are known, in the best case, up to a multiplicative constant. The aim of the slope heuristics method proposed by Birg\'e and  Massart~ \cite{birge2007minimal} is precisely to calibrate a penalty function for model selection purposes. See \cite{baudry2012slope} and ~\cite{arlot2019minimal} for a general presentation of the method. This method has shown very good performances and comes with mathematical guarantees for non parametric Gaussian regression with i.i.d. error terms,  see \cite{birge2007minimal,arlot2019minimal} and references therein. The slope heuristics have several versions (see ~\cite{arlot2019minimal}). In this paper we use the dimension jump algorithm, which is implemented for instance in the R package {\ttfamily capush}. 

The aim is to tune the constant $\kappa$ in a penalty of the form $\pen(m) = \kappa \pen_{\tiny \mbox{shape}}(m)$ where $\pen_{\tiny \mbox{shape}}$ is a known penalty shape. In the most standard cases, $\pen_{\tiny \mbox{shape}}$  is the dimension of the model. Let  $\hat{m}(\kappa)$ be the model selected by the penalized criterion with constant $\kappa$
\[\hat{m}(\kappa) \in \mathrm{argmin}_{m \in \mathcal{M}} \left\{ \frac1n  \left \| Y - \hat{f}_{m} \right \|_n^{2} +  \kappa \pen_{\tiny \mbox{shape}} (m) \right\}.\]
The Dimension Jump algorithm consists of the following steps (see Figure~\ref{fig:arma_contrast_dj} for an illustration)
\begin{enumerate}
\item  Compute $ \kappa  \mapsto  \hat{m}( \kappa )$,
\item Find the constant ${\hat \kappa}^{dj} > 0$ that corresponds to the highest jump of the function $\kappa \rightarrow d_{\hat{m}(\kappa)}$,
\item Select the model $\hat{m}(2 {\hat \kappa}^{dj})$, 
  $$\hat{m} \in \mathrm{argmin}_{m \in \mathcal{M}} \left\{ \left \| Y - \hat{f}_{m} \right \|_n^{2} + 2 {\hat \kappa}^{dj} \pen_{\tiny \mbox{shape}} (m)\right\} .$$
\end{enumerate}

\subsection{Presentation of the experiments} \label{subs:slopeh}

We simulate $n$ observations according to the following generative model on $[0,1]$
\begin{equation}
\label{eq:simu}
Y_i  =  f^\ast \left (\frac i n \right)   + \varepsilon_{i}, \quad   i = 1 \dots n  .
\end{equation}
In the simulations we take for  $ f^\ast $ the function
\begin{equation*}
f^\ast : t \in [0,1]   \mapsto   3 - 0.1*t + 0.5*t^2 - t^3 + \sin(8*t) .
\end{equation*}

The aim is to estimate $f^\ast $ on a regular partition of size $m$, for $m \in \{1, \ldots, 200\}$. 
We simulate $n$ observations $\varepsilon$ according to an ARMA  process, a Fractional Gaussian process and a non Gaussian Markov chain. The last framework allows us to evaluate the robustness of the model selection procedure without the Gaussian assumption.  We consider samples of size $n= 200$, $n= 500$, $n=2000$ , $n=5000$ and the risk of each regressogram is computed over $100$ simulations. 

\medskip

We now give more details on the error processes we use for the simulations.
\begin{itemize}
\item[•] \textbf{ARMA process.} The ARMA(2,1) short memory process is defined by
\begin{equation}
\varepsilon_{i} - 0.3 \varepsilon_{i-1} - 0.1 \varepsilon_{i-2} = W_{i} + 0.2 W_{i+1},
\label{arma21_process}
\end{equation}
where $(W_{i})_{i \in \mathbb{Z}}$ is a sequence of i.i.d. $\mathcal{N}(0,1)$ random variables.
\item[•] \textbf{Fractional Gaussian Noise.} 
The Fractional Gaussian Noise (FGN, see for instance \cite{MvN68} and \cite{Ber94}) is a stationary sequence $(\varepsilon_i)_{i \geq 1}$ of zero-mean Gaussian random variables with auto--covariances 
$$
\gamma_\varepsilon (k)= \frac{\sigma^2}{2} \left ( |k+1|^{2H}-2|k|^{2H} +|k-1|^{2H}  \right), \quad \text{for } k\in \mathbb N,
$$
where $\sigma^2= \gamma_\varepsilon(0)= \text{Var}(\varepsilon_i)$,  and $H \in (0, 1)$ is the so-called Hurst parameter. If $H=0$, the sequence $(\varepsilon_i)_{i \geq 1}$ is   a Gaussian white noise with variance $\sigma^2$. For any $H \in (0,1)$ the following asymptotic expansion is valid
$$
    \gamma_\varepsilon (k) \sim \sigma^2 H(2H-1) k^{2(H-1)} \, . 
$$
Consequently, if $H>1/2$, the process is positively correlated and long-range dependent. If $H<1/2$, the process is negatively correlated and   $\sum_{k\geq 0} |\gamma_\varepsilon (k) | < \infty$, so that \eqref{boundedradius} holds and the process is short-range dependent. 

In fact, for $H<1/2$, the FGN $(\varepsilon_i)_{i \geq 1}$   is anti-persistent in the sense of Definition \eqref{vsd} (with $\gamma= 2-2H$ in Definition \eqref{vsd}). This is well known (see for instance \cite{Ber94}), and follows from the fact that the $\varepsilon_i$'s are the increments of a fractional Brownian motion $B_H$, that is for $i=1, 2, \ldots$
$$
\varepsilon_i= B_H(i)-B_H(i-1),   \quad \text{with } \text{Var}(B_H(t))= \sigma^2 t^{2H}.
$$

In the simulations, we shall consider two cases
\begin{itemize}
\item[-] an anti-persistent case, with $H=0.2$,
\item[-] a long  memory case, with $H=0.7$.
\end{itemize}
\medskip

\item[•] \textbf{Non Gaussian Markov chain.} 
We start from the Markov chain introduced by Doukhan, Massart and Rio \cite{DMR94}. 

Let $a$ be a positive real number, 
let $\nu$ be the probability with density $x \rightarrow (1+a)x^a {\bf 1}_{[0,1]}$ and $\pi$ be the probability with density $x \rightarrow ax^{a-1}{\bf 1}_{[0,1]}$. We define now a strictly stationary Markov chain by specifying its transition
probabilities $K(x,A)$ as follows
\[ K(x,A)=(1-x)\delta_{x}(A)+x\nu(A)\, ,\]
where $\delta_{x}$ denotes the Dirac measure at point $x$. Then $\pi$ is the unique invariant probability measure of the chain with transition probabilities $K(x, \cdot)$. Let $(Z_i)_{i \in {\mathbb Z}}$ be the stationary Markov chain on $[0,1]$ with transition probabilities $K(x, \cdot)$ and invariant distribution $\pi$. Recall that the $\beta$-mixing coefficients of the chain $(Z_i)_{i \geq 1}$ are defined by 
$$
    \beta_Z(n) = \int \| K^n(x,\cdot)- \pi \|_v \,  \pi (dx), 
$$
where $\|\cdot \|_v$ is the variation norm. 
From  \cite{DMR94}, we know that 
 $\beta_Z(n)\sim \frac{1}{n^a}$. One can easily check than $Z_i^a$ is uniformly distributed over $[0,1]$, so that 
$$
\varepsilon_i= Z_i^a -0.5
$$
is a stationary Markov chain (as an invertible function of a stationary Markov chain), with mean zero and mixing coefficient $\beta(k) \sim \frac{1}{n^a}$. This chain is short range dependent if $a>1$ and long-range dependent if $a \in (0,1)$ (see for instance \cite{DGM18} for a deeper discussion on this subject). 

In the simulations, we shall consider three cases
\begin{itemize}
\item[-] two short memory cases, with $a=8$ and $a=1.5$,
\item[-] a long  memory case, with $a=0.5$.
\end{itemize}
\end{itemize}

In fact, for  regressograms on a regular partition of size $m$,  the main term of the penalty can be exactly determined by the behavior of Var($\varepsilon_1 + \cdots  + \varepsilon_n)$ (see the proof of  Lemma \ref{newLem}). More precisely, if 
\begin{equation*}
    \text{Var} \left (  \sum_{k=1}^n \varepsilon_k \right ) \sim \kappa n^{2- \gamma} \, ,
\end{equation*}
for some $\gamma \in (0, 2)$, then the  main term of the penalty will be of order $(m/n)^\gamma$.
We then see that $\gamma$ is related to the usual Hurst index $H$  (see for instance \cite{Ber94})  of the partial sum process
$$
S_n= \varepsilon_1 + \cdots + \varepsilon_n \, ,
$$
via the equality $\gamma= 2-2H$. Hence, for regressograms on a regular partition of size $m$, the main term of the penalty is of order  $(m/n)^{2-2H}$. 

This remains  true for estimators based on piecewise polynomial of degree $r \geq 1$ when $\gamma_\varepsilon (n) \sim  \kappa n^{- \gamma}$ for $\gamma \in (0, 1)$ : again the penalty is of order  $(m/n)^{2-2H}$ with 
$\gamma= 2-2H$ (see Subsection \ref{sub:case_long_range}). However for anti-persistent errors in the sense of \eqref{vsd}, the penalty term cannot be computed as precisely as for regressograms, and is of the usual order $m/n$ (as in the usual short range dependent case). \\

For long range dependent Gaussian processes, the variance terms of the risk are not linear  functions of the dimension, they behave as $m^\gamma$ for some $\gamma \in (0,1)$. Figure~\ref{fig:FG_compare_risk_shapes} shows the risk of the regressograms for observations simulated according to \eqref{eq:simu} with  the error process following a Fractional Gaussian distribution with Hurst exponents between $0.1$ and $0.9$. 

For anti-persistent cases ($H<0.5$), the risk has a convex behavior for large dimensions, in accordance with a variance term of order $m^{2-2H}$ (see Section~\ref{sub:case_very_short_range}). For the i.i.d. case ($H=0.5$), the risk is linear for high dimensions. For the long range dependent cases ($H>0.5$), the risk shows a concave behavior for large dimensions, in accordance with a variance term of order $m^{2-2H}$ (see Section~\ref{sub:case_long_range}).

\begin{figure}%
    \centering
\includegraphics[height=9cm, width=13cm]{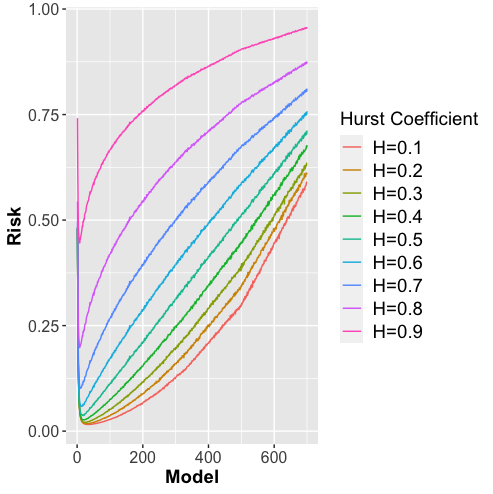}
 \caption{Comparison of risk shapes for the fractional Gaussian process with Hurst coefficient between 0.1 and 0.9, and for $n=2000$.}
    \label{fig:FG_compare_risk_shapes}%
\end{figure}

Figure~\ref{fig:chain_compare_risk_shapes} shows the risk of the regressograms for observations simulated according to~\eqref{eq:simu}, when the error process is the $\beta$-mixing Markov chains described above with a parameter $a$ between $0.3$ and $10$. We remark a concave behavior for long range dependent processes ($a<1$) and a linear behavior in the short range dependent case ($a>1$). This suggests that the theoretical results obtained in Sections~\ref{SRD} and~\ref{sub:case_long_range} could be also valid in non Gaussian contexts.\\

\begin{figure}%
    \centering
\includegraphics[height=9cm, width=13cm]{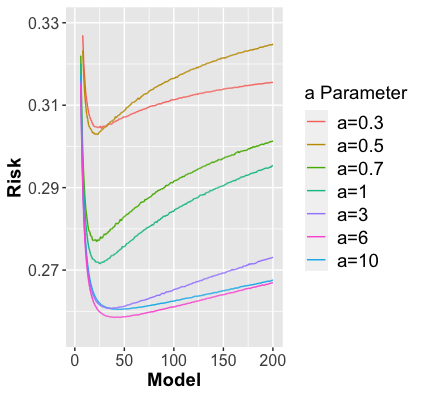}
 \caption{Comparison of risk shapes for the Markov chain, for $n=2000$.}
    \label{fig:chain_compare_risk_shapes}%
\end{figure}

\medskip

For the simulations, we use the Whittle MLE-estimator~\cite{whittle1953estimation} 
implemented in the {\ttfamily longmemo} package, to estimate the Hurst index $H$. We compare several approaches
\begin{itemize}
\item {\bf CDJ}: Classical Dimension Jump method with a penalty shape proportional to the dimension.
\item {\bf HGiven}: Dimension Jump for the penalty shape $m^{2-2H}$ with Hurst exponent $H$ given.
\item {\bf Wh(Y)}: Dimension Jump for the penalty shape $m^{2-2 \hat H}$ where $\hat H$ is the Whittle estimator 
computed on the $Y$ process.
\item {\bf Wh(Res)}: Dimension Jump for the penalty shape $m^{2 - 2 \hat H}$ where $\hat H$ is the Whittle estimator 
computed on the residuals of a model. \newline
\noindent For the method Wh(Res), we have to propose a model $m_0$ for which the Hurst exponent is computed on the residuals. Roughly speaking, the idea is to estimate the Hurst exponent in a sufficiently large model for which the bias is negligible.  
We propose a two step procedure, which is based on the selection of a pre-model $\hat m_1$  to estimate the Hurst exponent $H$ on the residuals of $\hat m_1$. This provides an estimator $\hat H$ which is used to design the penalty shape. The dimension jump is then used to select the final model $\hat m$. We propose two versions for this two-step procedure:
\begin{itemize}
\item[-] \textbf{CDJ+Wh(Res)}: Classical Dimension Jump to find a pre-model $\hat m_1$, then Whittle estimator $\hat	 H$ to estimate the Hurst exponent and finally  Dimension Jump with penalty shape $m^{2-2 \hat H}$.
\item[-] \textbf{Wh(Y)+Wh(Res)}: Dimension Jump with penalty shape $m^{2 - 2 \hat H_1}$ where $\hat H_1$ is the Whittle estimator on $Y$, this selects a pre-model $\hat m_1$, then Whittle estimator $\hat H_2$ on the residuals of the model $\hat m_1$ and finally  Dimension Jump with penalty shape $m^{2-2 \hat H_2}$.
\end{itemize}
\end{itemize}
 
\subsection{Short range dependence}
 

In this section we study the performance of the model selection method in the short dependence framework. The penalty shape is chosen proportional to the model dimension, as in the i.i.d. case and we can apply the classical dimension jump method (CDJ) to calibrate $\kappa$. 
\noindent Roughly speaking, the slope heuristics relies, among other assumptions, on the fact that the empirical contrast behaves in high dimension as a linear function of the penalty shape. 

We also compare the performances of the CDJ method with the ones of the other approaches. As we shall see, other methods can give better results for $n$ small.\\

$\bullet$ \textbf{Gaussian ARMA process}\\

We begin with the classical ARMA(2,1) short memory process defined in~\eqref{arma21_process}.
Figure~\ref{fig:arma_contrast} shows the behavior of the empirical contrast for $n = 2000$ and an illustration of the dimension jump algorithm. As expected by the slope heuristics, a  linear behavior of the empirical contrast can be observed in high dimensions ($m \geq 25$).
 \begin{figure}%
    \centering
    \subfloat[Linear behavior of the empirical contrast ($n=2000$).]
    {\label{fig:arma_contrast_linear}  \includegraphics[scale=0.47]{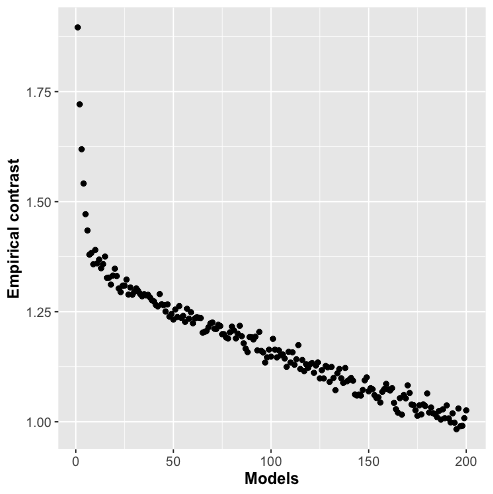}}
    \subfloat[Dimension Jump.]
    { \label{fig:arma_contrast_dj} \includegraphics[scale=0.47]{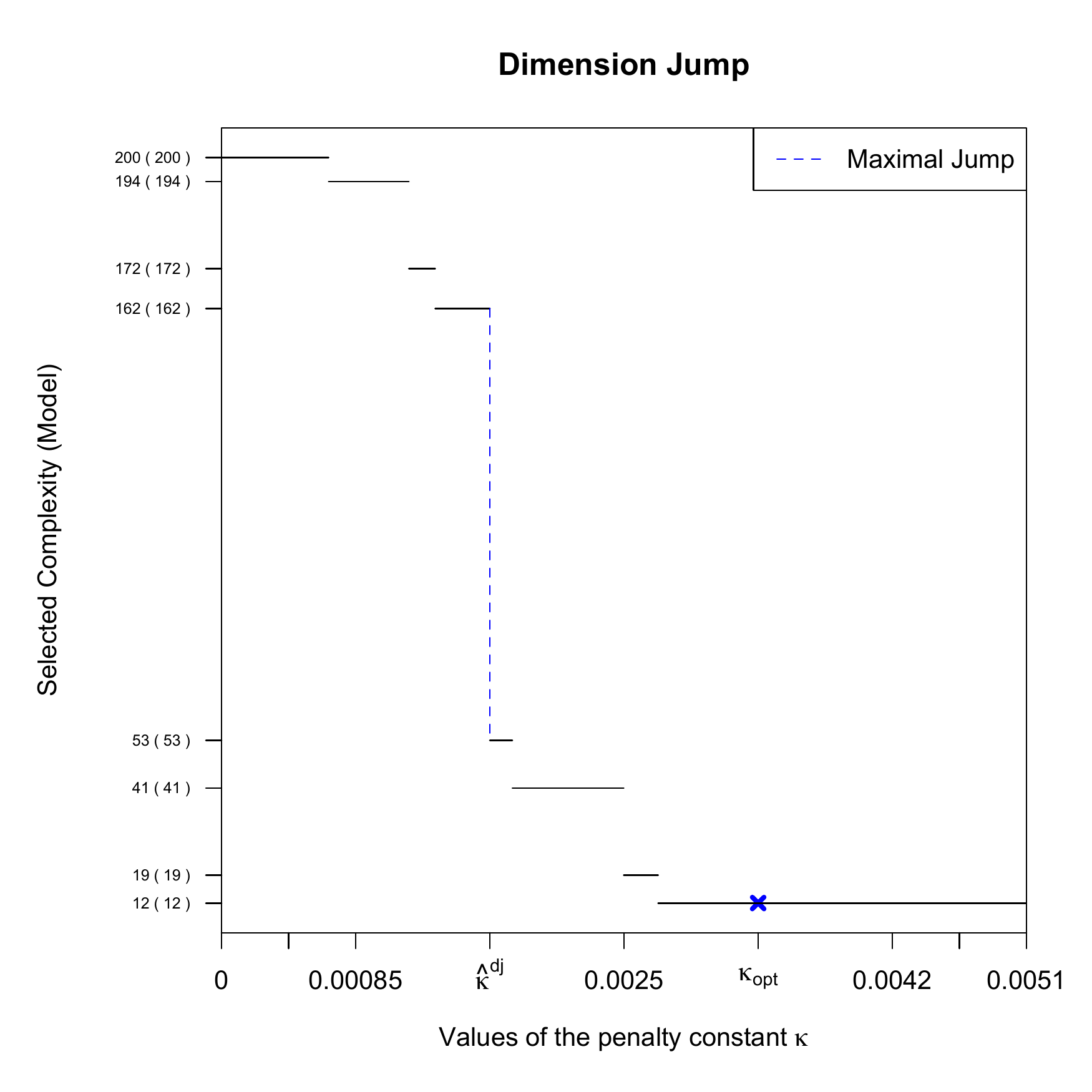}}      
 \caption{Illustration of the slope heuristics for the ARMA(2,1) process.}%
    \label{fig:arma_contrast}%
\end{figure}

Figure~\ref{fig:arma2_1_risk} shows the performance of the different methods. The boxplots on the left part of each graph show the risk of this model selection method over 100 trials. On the right, the risk function is displayed.

In this experiment, the classical dimension jump (penalty shape proportional to the dimension) works clearly well for $n$ large ($n \geq 2000$). It is however less efficient for $n$ small.  Indeed,  the risk shows a concave behavior in large dimensions, as in the long memory case (as we shall see later on). For small $n$, an estimation of $H$ with the Whittle estimator applied on the $Y$ process and plugged into the penalty shapes finally gives better results than the classical dimension jump method.

The Whittle estimator computed on the residuals is also efficient for selecting the minimal risk model for $n$ small. In this case we consider the residuals process of the model chosen at first step either by CDJ or by Wh(Y), the method CDJ + Wh(res) having bad results for $n$ too small ($n=200$).\\

\begin{figure}%
    \centering
    \subfloat[$n=200$]{{\includegraphics[width=0.49\textwidth]{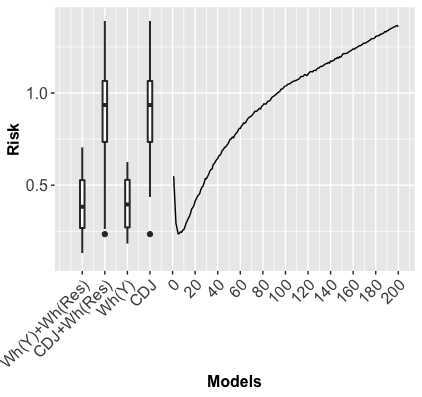} }}%
    \subfloat[$n=500$]{{\includegraphics[width=0.49\textwidth]{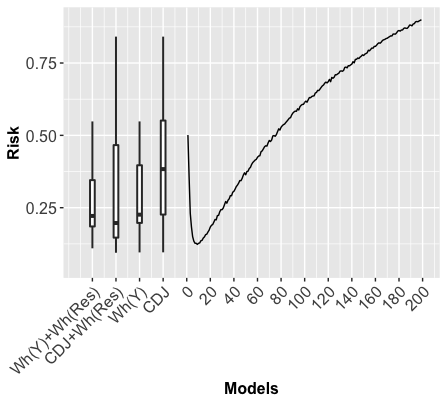} }} \\
    \subfloat[$n=2000$]{{\includegraphics[width=0.49\textwidth]{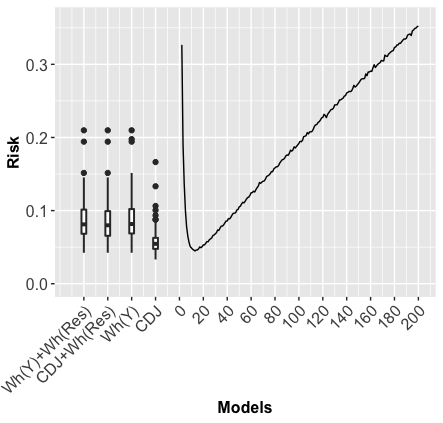} }}%
    \subfloat[$n=5000$]{{\includegraphics[width=0.49\textwidth]{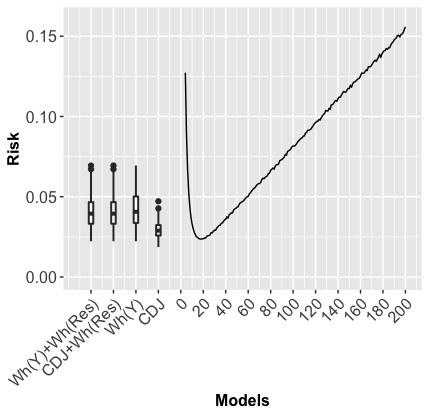} }}%
    \caption{Short Memory ARMA process. Risk curves and performances of the different calibration  methods for $n = 200, 500, 2000, 5000$.}%
    \label{fig:arma2_1_risk}%
\end{figure} 

$\bullet$ \textbf{Non Gaussian Markov chain}\\

To evaluate the robustness of the model selection procedure without the Gaussian error assumption, we consider  the Non Gaussian Markov chain defined above. We simulate an error process $\varepsilon$ distributed according to this stationary Markov chain, and 
we first make simulations in the short dependent case with a value of $a = 8$. As shown by Figure~\ref{fig:chain8_contrast}, a linear behavior of the empirical contrast can be observed, which is a good point for applying the slope heuristics here. 

The performances of the methods are summarized on Figure~\ref{fig:chain8}. We can check on this figure that the classical dimension jump shows good performances. 
For all  sample sizes, the dimension jump based on the Whittle estimator applied to $Y$ is a little less efficient than the two-step methods.  

 \begin{figure}
    \centering
    \subfloat[Linear behavior of the empirical contrast.]{
     { \includegraphics[scale=0.47]{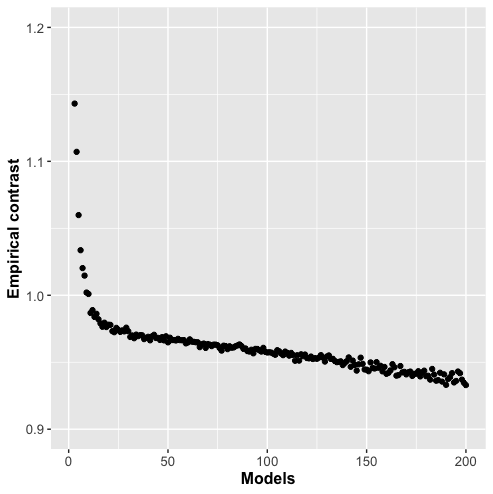}}}
    \subfloat[Dimension Jump.]{{\includegraphics[scale=0.47]{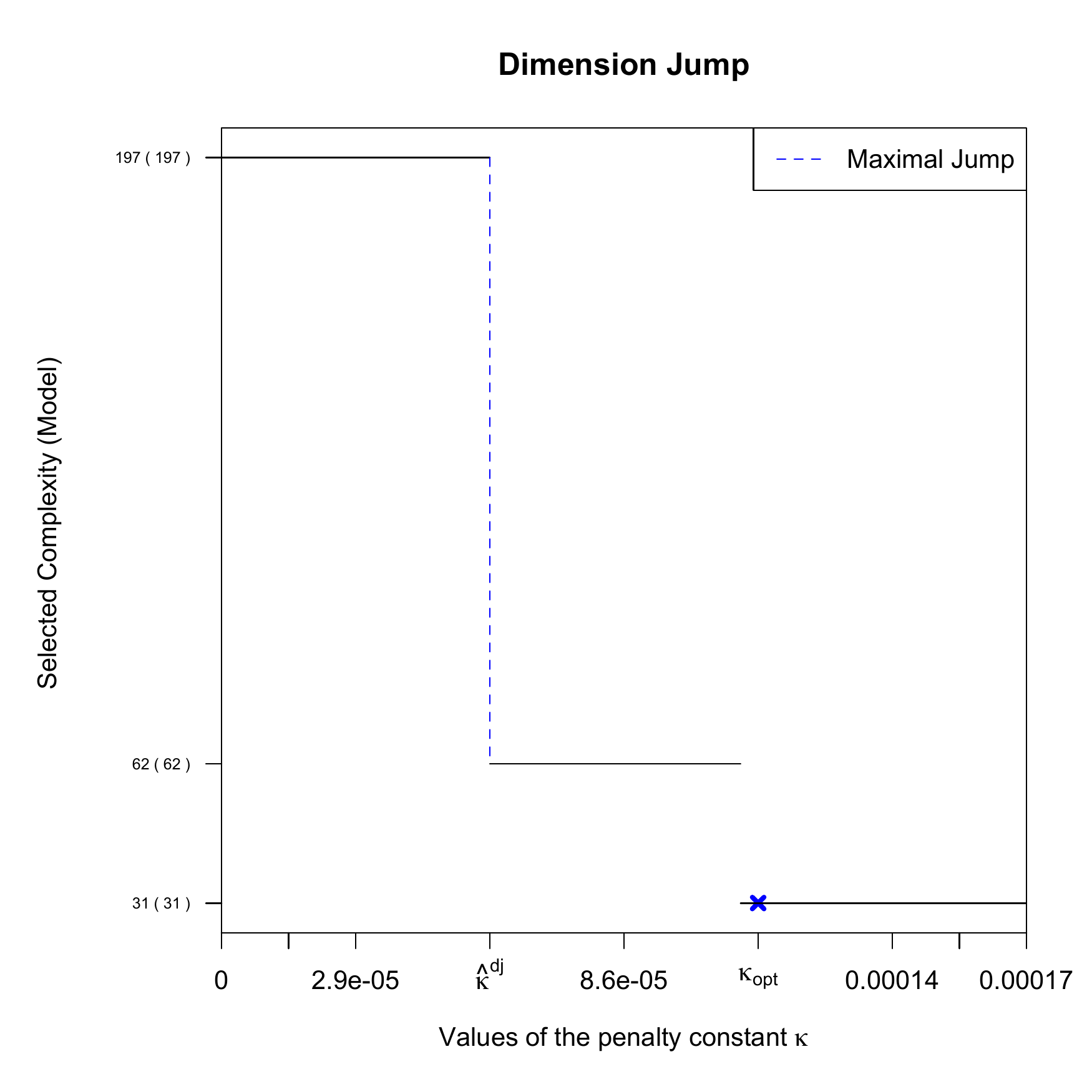}}}
 \caption{Illustration of the slope heuristics for the non Gaussian process ($a = 8$).}%
    \label{fig:chain8_contrast}%
\end{figure}

\begin{figure}
    \centering
    \subfloat[$n=200$]{{\includegraphics[width=0.49\textwidth]{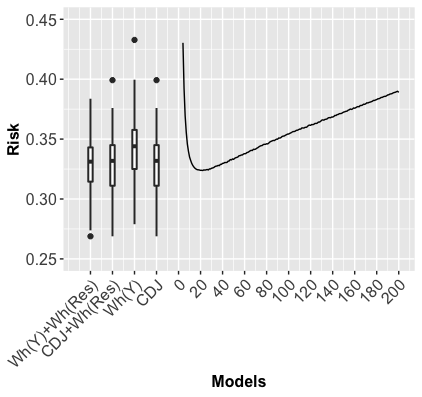} }}%
    \subfloat[$n=500$]{{\includegraphics[width=0.49\textwidth]{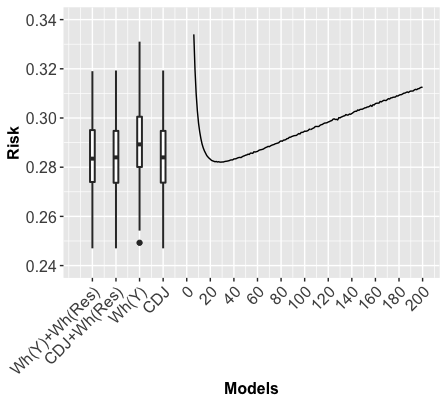} }} \\
    \subfloat[$n=2000$]{{\includegraphics[width=0.49\textwidth]{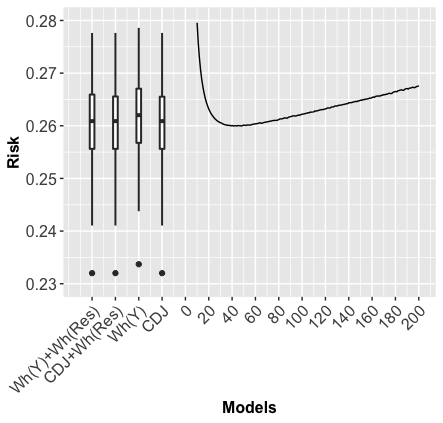} }}%
    \subfloat[$n=5000$]{{\includegraphics[width=0.49\textwidth]{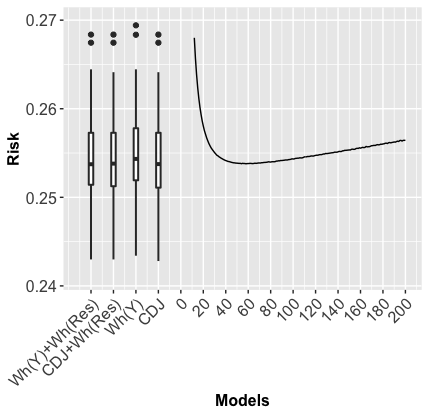} }}%
    \caption{Markov chain with $a=8$.  Risk curves and performances of the different calibration  methods for $n = 200, 500, 2000, 5000$.}%
    \label{fig:chain8}
\end{figure}
 
We now consider a second short memory case with the Markov chain,  with $a=1.5$. This case is very closed to the limit case $a=1$, which separates long memory from short memory. Figure~\ref{fig:chain1dot5} shows that the CDJ method works well for $n$ large. But for $n$ small, the four methods do not really manage to select a model close to the oracle model.

The methods based on the direct estimation of the Hurst exponent, like Wh(Y), give good results for $n$ smaller than $500$. 
Regarding the two-step methods, CDJ+Wh(res) shows bad performances for $n$ small ($n \leq 500$), while Wh(Y)+Wh(res) shows good results for $n=500$ but poor results for $n=200$.
 
\begin{figure}%
    \centering
    \subfloat[$n=200$]{{\includegraphics[width=0.49\textwidth]{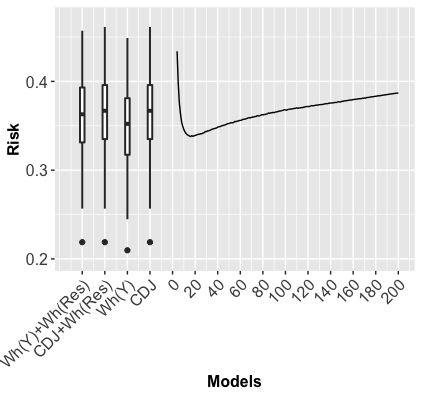} }}%
    \subfloat[$n=500$]{{\includegraphics[width=0.49\textwidth]{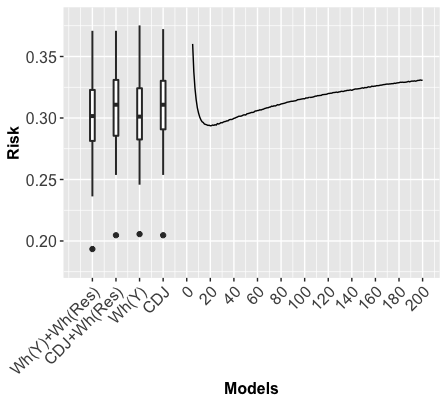} }} \\
    \subfloat[$n=2000$]{{\includegraphics[width=0.49\textwidth]{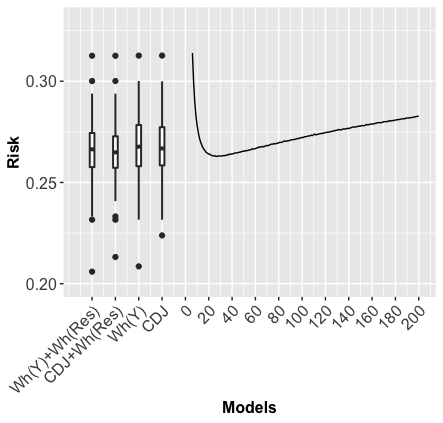} }}%
    \subfloat[$n=5000$]{{\includegraphics[width=0.49\textwidth]{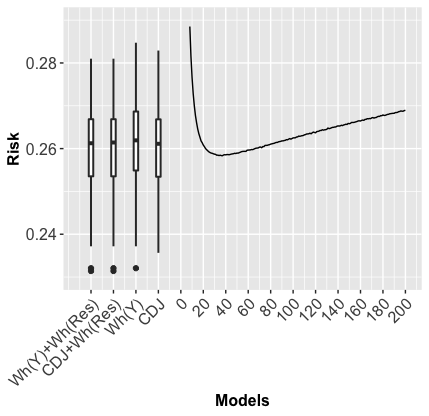} }}%
    \caption{Markov chain with $a=1.5$. Risk curves and performances of the different calibration  methods for $n = 200, 500, 2000, 5000$.}%
    \label{fig:chain1dot5}%
\end{figure} 
 
\subsection{Long range dependence}


For long range dependent Gaussian processes, the variance terms of the risk are not linear functions of the dimension, they behave as  $m^\gamma$ for some parameter $\gamma \in (0,1)$. We thus would like to use penalties proportional to $m^\gamma$, 
see Section~\ref{sub:case_long_range}. For instance, for Fractional Gaussian processes, $\gamma = 2-2H$, where $H$ is the Hurst exponent. Of course this coefficient is unknown in practice and thus we use some estimator of the Hurst exponent to calibrate the penalty. Generally speaking, estimating the Hurst exponent is a difficult statistical task, however a rough estimation can be sufficient for the model selection problem we study here.

\vskip 0.5cm

$\bullet$ \textbf{Fractional Gaussian Noise}\\
 
For this experiment we simulate the error process with a Gaussian Fractional Noise of Hurst parameter $H=0.7$.  The performances of the methods are summarized on Figure~\ref{fig:FGH0dot7}.  
We can check on this figure that when using a penalty with the true Hurst exponent ($H=0.7$) of the error process, the model selection method works correctly. We also note that the classical dimension jump (penalty shape proportional to the dimension) shows bad performances. 
On the other hand, the Whittle estimators applied to $Y$ and plugged into the penalty shape show good results for all samples size. The two steps methods show also good performances for $n$ large enough.\\

 \begin{figure}%
    \centering
    \subfloat[$n=200$]{{\includegraphics[width=0.49\textwidth]{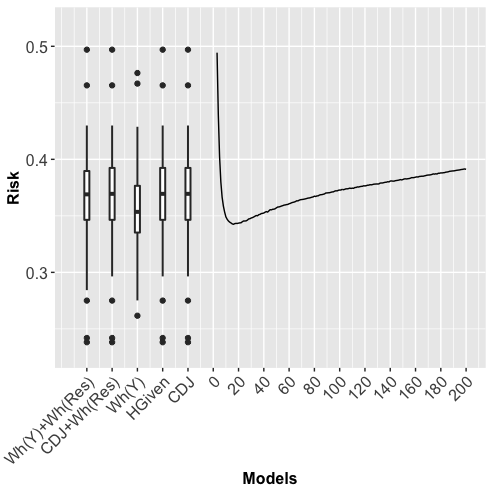} }}%
    \subfloat[$n=500$]{{\includegraphics[width=0.49\textwidth]{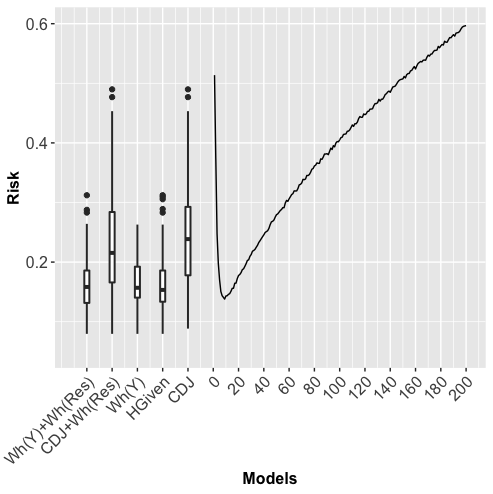} }} \\
    \subfloat[$n=2000$]{{\includegraphics[width=0.49\textwidth]{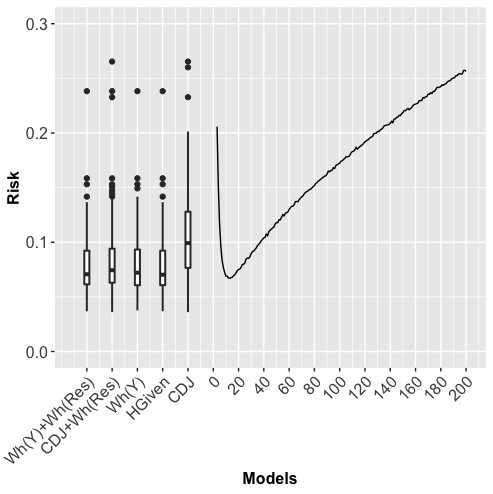} }}%
    \subfloat[$n=5000$]{{\includegraphics[width=0.49\textwidth]{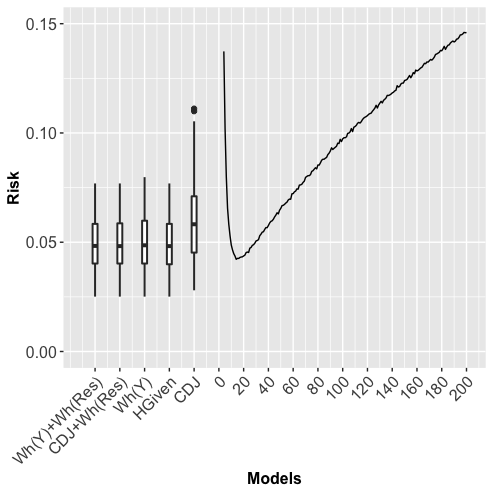} }}%
    \caption{Long Memory Fractional Gaussian error process with $H = 0.7$. Risk curves and performances of the different calibration methods for $n = 200, 500, 2000, 5000$.}%
    \label{fig:FGH0dot7}%
\end{figure}

\medskip 

$\bullet$ \textbf{Non Gaussian Markov chain}\\

We now evaluate the robustness of our model selection procedure when the Gaussian error assumption is not satisfied. We consider here the Non Gaussian Markov chain in the long range dependent setting.
As for the Fractional Gaussian Noise, the risk has a concave behavior for large dimension, see  Figure~\ref{fig:chain_compare_risk_shapes} for an illustration. Then the penalty shape is equal to $m^{a}$, where $a$ is the decay rate of the covariances.

For this experiment we simulate the Markov chain with $a = 0.5$ for the error process. The performances of the methods are displayed on Figure~\ref{fig:chaindot5}. 
We observe that the classical dimension jump shows bad performances in this non Gaussian long range dependent context. When using the penalty shape $m^a$ ($H$ given, with $a=2-2H)$, the performances are a little better than before, but not as good as one could hoped for. 
For $n$ large enough ($n \geq 2000$), the Whittle estimators applied on $Y$ and plugged into the penalty shape shows satisfactory results. The performances of the two step methods are similar but from $n=5000$.

This experiment suggests that more work should be done in this context. It seems that a concave penalty shape should be used, as expected, but that the good exponent could perhaps be different from $a=2-2H$. 
 
 \begin{figure}%
    \centering
    \subfloat[$n=200$]{{\includegraphics[width=0.49\textwidth]{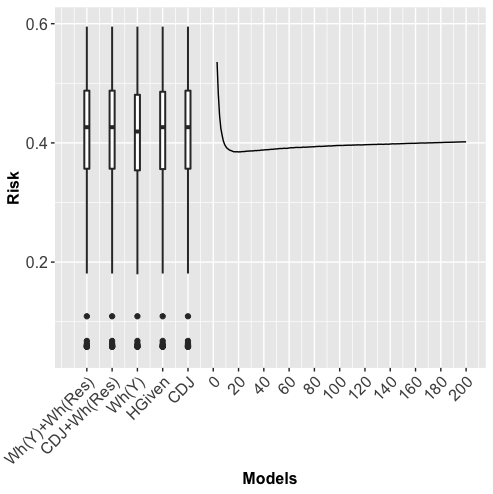} }}%
    \subfloat[$n=500$]{{\includegraphics[width=0.49\textwidth]{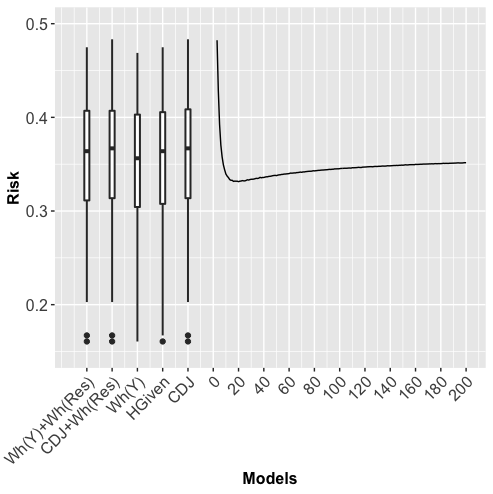} }} \\
    \subfloat[$n=2000$]{{\includegraphics[width=0.49\textwidth]{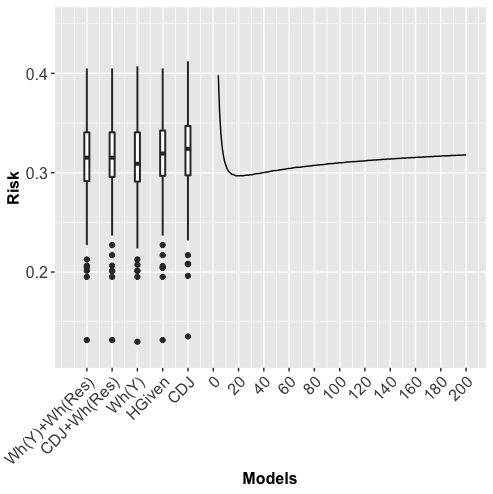} }}%
    \subfloat[$n=5000$]{{\includegraphics[width=0.49\textwidth]{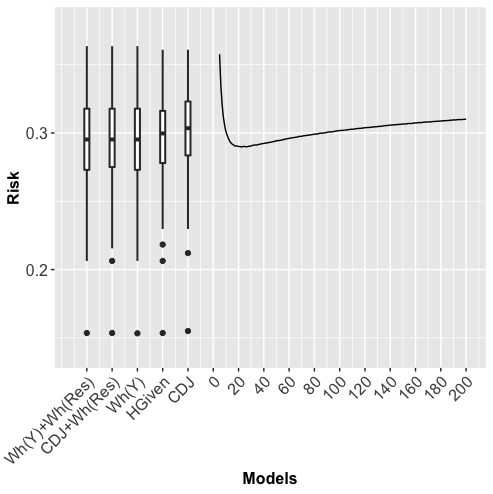} }}%
    \caption{Markov chain process with $a = 0.5$. Risk curves and performances of the different calibration  methods for $n = 200, 500, 2000, 5000$.}%
    \label{fig:chaindot5}%
\end{figure}

\subsection{Anti-persistent errors with a Fractional Gaussian Noise}

We consider the same simulation protocole with anti-persistent errors, following a Fractional Gaussian Noise with $H = 0.2$. Again, we observe a linear behavior of the empirical contrast in high dimension, see Figure~\ref{fig:FGH0dot2_contrast_linear}.
 \begin{figure}%
    \centering
    \subfloat[Linear behavior of the empirical contrast for $n=2000$.]{
    \label{fig:FGH0dot2_contrast_linear}
     \includegraphics[scale=0.47]{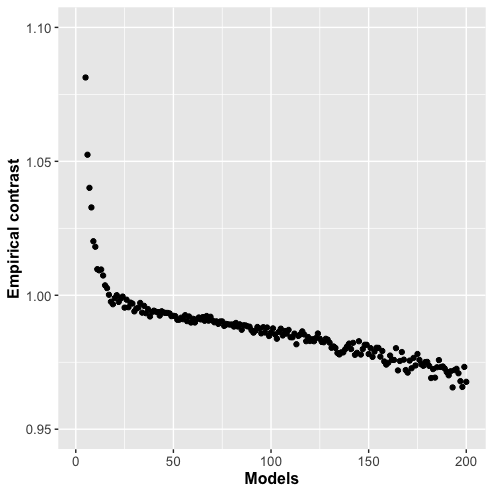}
     }
    \subfloat[Dimension Jump.]{{\includegraphics[scale=0.47]{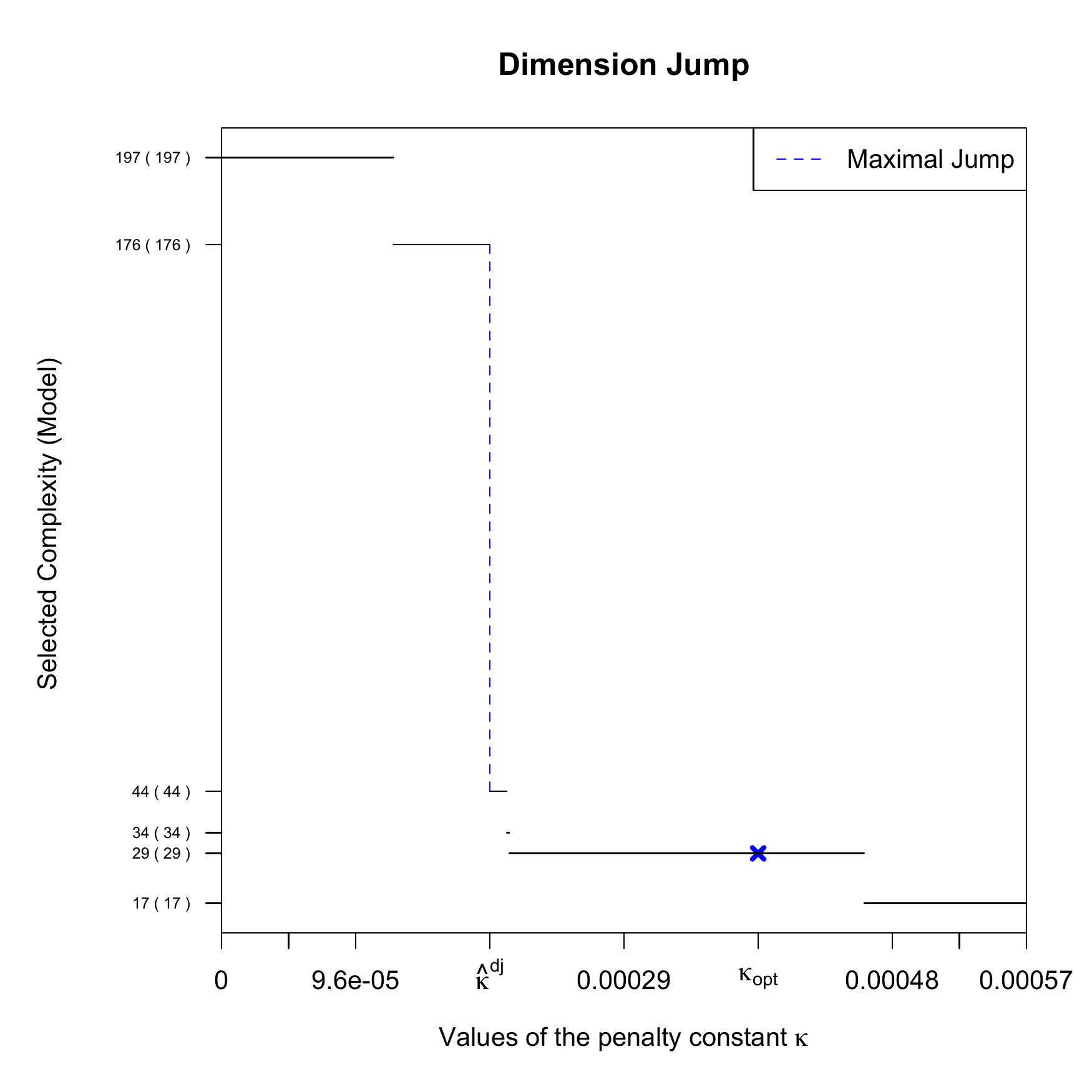}}}
 \caption{Illustration of the slope heuristics for the Fractional Gaussian process ($H = 0.2$).}%
    \label{fig:FGH0dot2_contrast}%
\end{figure}

The performances of the different methods on this experiment are summarized by Figure~\ref{fig:FGH0dot2}.
We can check that when using a penalty with the true Hurst exponent ($H=0.2$), the model selection method works pretty well. 
The two-step methods, with the Whittle estimator computed on the residuals, give similar results for all $n$. 
On the other hand, the Whittle estimator applied directly on $Y$ shows poor performances for $n$ small, but it is better for $n$ large. 

We also note that, in this short range dependent case, the classical dimension jump shows good results for all $n$, as in the i.i.d. case.

 \begin{figure}%
    \centering
    \subfloat[$n=200$]{{\includegraphics[width=0.49\textwidth]{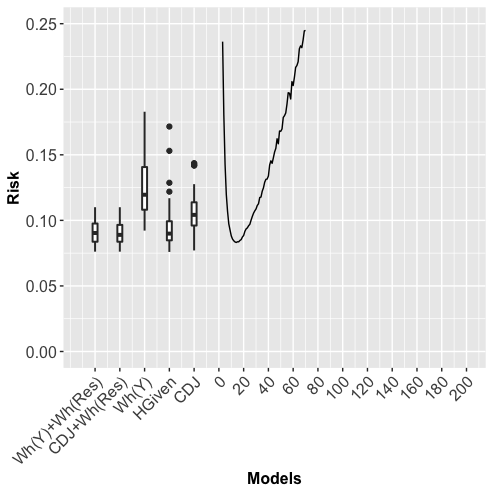} }}%
    \subfloat[$n=500$]{{\includegraphics[width=0.49\textwidth]{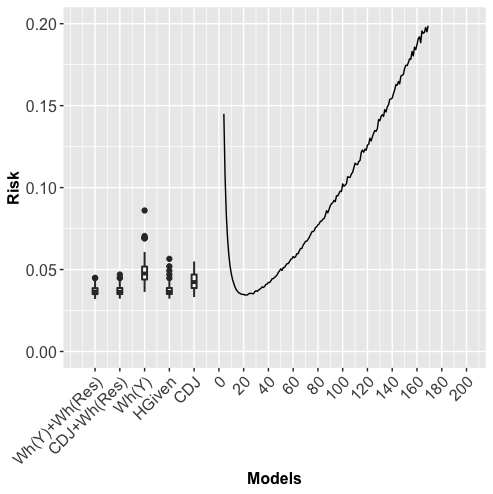} }} \\
    \subfloat[$n=2000$]{{\includegraphics[width=0.49\textwidth]{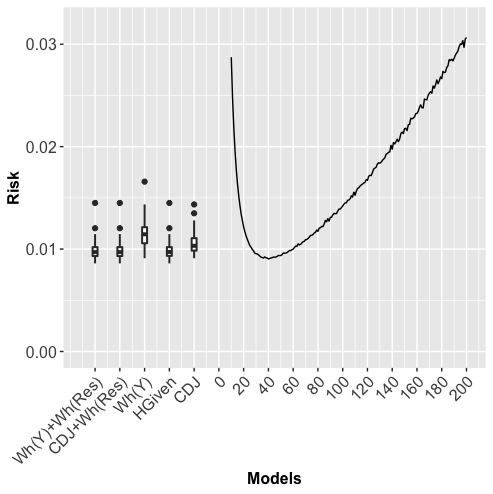} }}%
    \subfloat[$n=5000$]{{\includegraphics[width=0.49\textwidth]{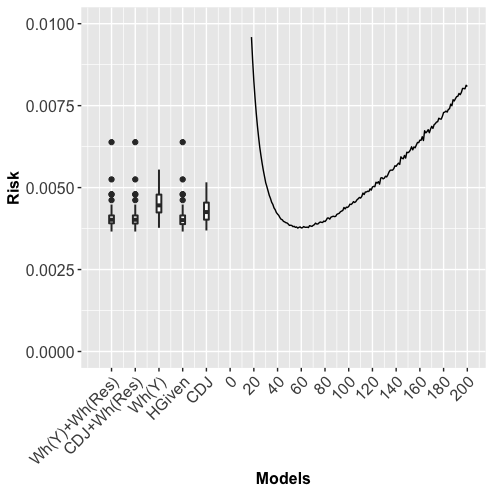} }}%
    \caption{Short Memory Fractional Gaussian process with $H = 0.2$. Risk curves and performances of the different calibration  methods for $n = 200, 500, 2000, 5000$.}%
    \label{fig:FGH0dot2}%
\end{figure}

\subsection{Conclusion on the experiments}


In these experiments we see that the penalty proportional to $(m/n)$ (with a constant calibrated thanks to the  jump dimension algorithm: CDJ method) performs quite well for short memory processes, but underperforms in all the other situations. The Wh(Y) method, with a penalty proportional to $(m/n)^{2-2\hat H}$ and an estimator $\hat H$ based ont he $Y_i$'s, performs quite well in most of the cases, but can show very bad performances (see for instance Figure~\ref{fig:FGH0dot2}) and is hard to justify from a heuristic point of vue.  The two steps methods, with a penalty proportional to $(m/n)^{2-2\hat H_2}$ and an estimator $\hat H_2$ based on the residuals of the first adjustment, performs well in most of the cases, with a clear preference for the Wh(Y)+Wh(Res) method. In fact, we suspect an overfitting with method CDJ for long memory processes, so that the residuals based on CDJ are not close to the original  error process (see the application to the Nile data in Section~\ref{sec::data_Nil}). 

We note that the two step method  Wh(Y)+Wh(Res) gives performances close, even sometimes better, to the best of the other proposed methods. An interesting example is the Gaussian ARMA process: for large $n$ ($n \geq 2000$), the risk curve is quasi linear, and the CDJ method is the best method. But for small $n$ ($n \leq 500$), the risk curve is concave, as in the long memory case, and the Wh(Y)+Wh(Res) is the best method.  This suggests that, even for short memory processes, a penalty proportional to $(m/n)$ is not always a wise choice in practice. 

Our final  comment is then: instead of looking for a penalty proportional to $(m/n)^\gamma$ for an appropriate $\gamma$, it might be preferable to estimate directly the term $\tr (\Proj_{S_{m}} \Sigma)$. This could perhaps be done by giving an estimation of the covariance $\Sigma$ based on the residuals of an appropriate pre-model. 

\section{Application to Nile data}
\label{sec::data_Nil}

In this section, we wish to continue the discussion on the Nile data initiated by Robinson in his 1997 article \cite{Rob1997Dep}. We borrow from Robinson his presentation of this dataset, as well as some other sentences: "These data consist of readings of annual minimum levels at the Roda gorge near Cairo, commencing in the year 622; often only the first
663 observations are employed because missing observations occur after the year 1284 (see~\cite{toussoun1925memoire}). It was one of the hydrological series examined by~\cite{hurst1951long} which led to his recognition of the "Hurst effect" and invention of the $R/S$ statistic". The data are plotted in Figure~\ref{fig:nileriverdata}. 

Robinson then summarizes the different ways of apprehending these data: either by considering that the cyclical variations come from a phenomenon of long memory, or by considering that the series can be written as the sum of a deterministic tendency plus a random noise. We refer to his article for relevant references on these questions.

Robinson  applied different kernel estimators (with different bandwidths) to estimate the regression function. Then he estimated the Hurst coefficient $H$ of  the errors from the residuals of the regression (see Section 4 of his paper for the definition of the estimator of $H$). He noted that "These estimates thus vary greatly over the ranges of the smoothing employed" and concluded this section by "This study highlights the need for developing methods for choosing $b$ and $c$ which respond automatically to the strength of the dependence in $u_t$" (here $b$ and $c$ are the bandwidth used  to estimate the regression function and the Hurst index respectively; $u_t$ is the error process, according to Robinson's notations).

This last sentence motivates us to apply our methods on these data, since we have a way to select automatically a partition from the data. We try two penalties: the usual penalty proportional to $m/n$, using the "classical jump dimension" to calibrate the constant (see CDJ method in Section~\ref{sec::simus}); this method should work well if the underlying error process was short range dependent. And a penalty proportional to $(m/n)^{2 -2\hat H_2}$, where $\hat H_2$  is the Hurst estimator based on the residuals, according to the Wh(Y)+Wh(Res) method described in Section~\ref{sec::simus}.  Indeed, this method was the best method according to  the different kind of simulations done in Section~\ref{sec::simus}. 
The resulting estimators are plotted in Figure~\ref{fig:nileriver_regression}. 

The CDJ method selects a partition of size $m=54$, with a clear impression of overfitting: the estimated trend seems very irregular, with many brutal changes. It seems that some randomness is still present in the trend. The Hurst index estimated through the residuals obtained with the estimated trend gives $\hat H=0.59$, hence not so far from a white noise.

The Wh(Y)+Wh(Res) selects a much smaller partition, with $m=7$. The trend looks more regular and interpretable, with a clear minimal period, a clear maximal period, and an almost constant tendency in between. It also suggests that an irregular partition should be used, which is a priori doable with our model-selection method, at the price of more tricky computations and algorithms. The Hurst index estimated through the residuals obtained with the estimated trend gives $\hat H=0.79$,  in accordance with the long-range dependence hypothesis. 

To be complete, the graph and the ACF of the residuals obtained with the Wh(Y) + Wh(Res) method are plotted in Figure~\ref{fig:nileriver_resacf_method_why_whres}.

\begin{figure}%
    \centering
	\includegraphics[width=0.49\textwidth]{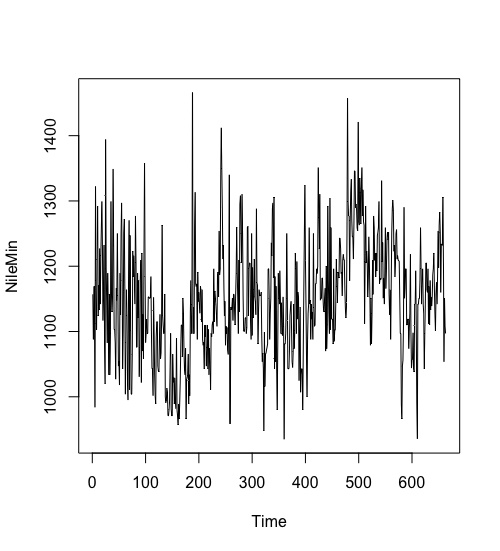}
	\caption{Nile River data.}%
    \label{fig:nileriverdata}%
\end{figure}

\begin{figure}%
    \centering
    \subfloat[Regressogram with CDJ]{{\includegraphics[width=0.49\textwidth]{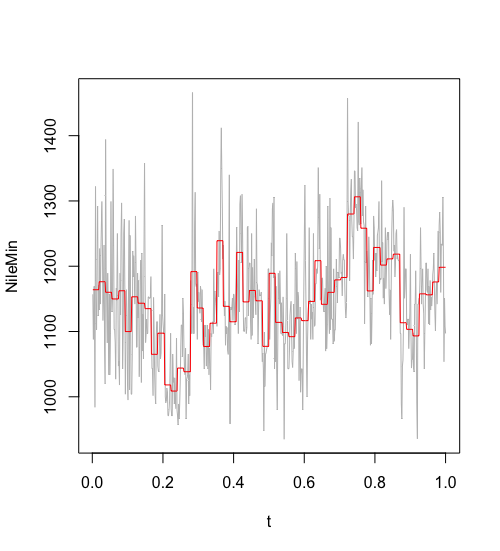} }}%
    \subfloat[Regressogram with Wh(Y)+Wh(res)]{{\includegraphics[width=0.49\textwidth]{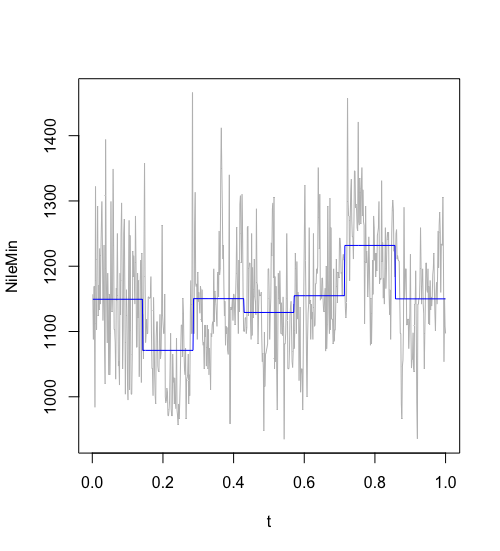} }} \\
    \caption{Nile River data and resulting estimators.}%
    \label{fig:nileriver_regression}%
\end{figure}

\begin{figure}%
    \centering
    \subfloat[Residuals]{{\includegraphics[width=0.49\textwidth]{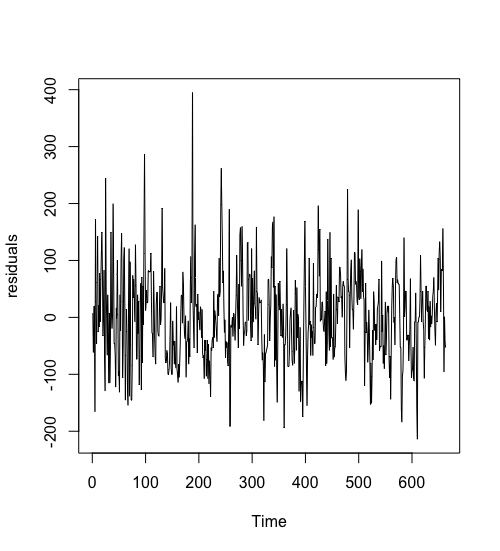} }}%
    \subfloat[ACF]{{\includegraphics[width=0.49\textwidth]{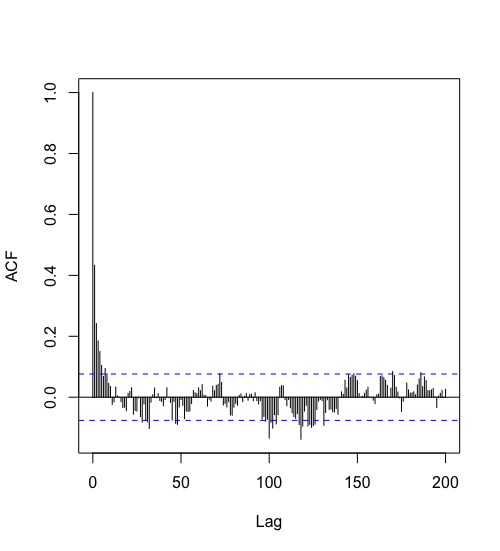} }} \\
    \caption{Residuals and ACF of the residuals for the method Wh(Y)+Wh(res).}%
    \label{fig:nileriver_resacf_method_why_whres}%
\end{figure}




\section{Discussion}
\label{sec::discuss}

This paper deals with linear model selection  with Gaussian dependent errors through  $\ell_0$ penalization. Several generalizations and extensions could be proposed in future works.

In this paper, we apply Theorem~\ref{theo:main_theorem} to study the fixed design case, but clearly the theorem  also  applies to all the settings considered in \cite{birge2001gaussian} (or Chapter~2 in~\cite{giraud2014introduction}) in the i.i.d case. In particular, if the error process is short range dependent, then for all these problems the penalty is the same as the i.i.d. case, the usual variance term being replaced by the spectral radius of the covariance matrix.

The performances of the $\ell_0$ penalization strategy are studied in this work  assuming that the distribution of the errors is stationary. However, Theorem~\ref{theo:main_theorem} does not require this assumption. 
In a similar line of work, \cite{gendre2008simultaneous} considers model selection for heteroscedastic Gaussian regression, for independent observations. It would be  possible to  study model selection for heteroscedastic Gaussian linear models with dependence and in particular in the long memory setting.  

An other line of research concerns  an extension of Theorem~\ref{theo:main_theorem} for non linear models. Indeed, in the independent setting, a general model selection for non linear models is given  in~\cite{massart2007concentration} (Theorem 4.18). By combining a Gaussian concentration inequality together with a chaining argument for dependent variables, we believe that it is possible to generalize the $\ell_0$ penalization strategy  for non linear models.
 
Our work strongly relies on the Gaussian assumption. It would be also interesting to provide model selection results for non Gaussian noise. Note that \cite{gendre2014model} gives a general model selection theorem for linear models, under moment conditions. It would be interesting to revisit these results in the context of long range dependence. 
 
As illustrated in the last sections,  it appears to be possible to adapt the slope heuristics for calibrating penalties in the context of regression with dependent errors. It would be more satisfactory to provide justification of the slope heuristics in this context. A first step would be to justify the slope heuristics for regression with short memory errors. 
Finally, note that model selection for density estimation under mixing conditions with resampling penalties has been studied in \cite{lerasle2011optimal}. This strategy is computationally expensive but it  deserves to be investigated for regression under short and long memory errors.

\subsection*{Acknowledgment}

The authors are grateful to Anne Philippe for helpful discussions and suggestions about statistics of long memory processes.

\section{Proofs}
\label{sec::proofs}

\subsection{Proof of Theorem~\ref{theo:main_theorem}}

We adapt the proof of Theorem~$2.2$ in~\cite{giraud2014introduction} in the framework of dependent Gaussian errors. Starting from the definition of $\hat{m}$, see Equation~\eqref{crit}, we find that for all $m \in \mathcal{M}$
\[\left \| Y - \hat{t}_{\hat m} \right \|_n^{2} + \pen(\hat{m}) \leq \left \| Y - \hat{t}_{m} \right \|_n^{2} + \pen(m).\]
Next,
\[\left \| \varepsilon + (t^{\ast} - \hat{t}_{\hat m}) \right \|_n^{2} + \pen(\hat{m}) \leq \left \| \varepsilon + (t^{\ast} - \hat{t}_{m}) \right \|_n^{2} + \pen(m),\]
and thus
\[\left \| \varepsilon \right \|_n^{2} + \left \| t^{\ast} - \hat{t}_{\hat m} \right \|_n^{2} + 2 \langle \varepsilon,t^{\ast} - \hat{t}_{\hat m} \rangle_n + \pen(\hat{m}) \leq \left \| \varepsilon \right \|_n^{2} + \left \| t^{\ast} - \hat{t}_{m} \right \|_n^{2} + 2 \langle \varepsilon,t^{\ast} - \hat{t}_{m} \rangle_n + \pen(m),\]
where $\langle \cdot, \cdot \rangle_n$ is the normalized inner product in $\R^n$: $\langle \cdot, \cdot \rangle_n =  \frac 1n\langle \cdot, \cdot \rangle $.
It can be checked that $\E \left[ \langle \varepsilon,t^{\ast} - \hat{t}_{m} \rangle_n \right] \leq 0 $ and finally we obtain that
\[
 \E  \left\| t^{\ast} - \hat{t}_{\hat m} \right \|_n^{2} \leq  \E  \left \| t^{\ast} - \hat{t}_{m} \right \|_n^{2}  + \pen(m) + 2  \E  \left( \langle \varepsilon,\hat{t}_{\hat m} - t^{\ast} \rangle_n   - \pen(\hat{m}) \right).
\]
The theorem can be directly derived from the next result

\begin{prop} \label{prop:pen_Z}
For the penalty defined by Equation~\eqref{eq:pen0}, there exists  some constants $a > 1$ and $L_{K} \geq 0$  that only depend on $K$, and a random variable $Z$ satisfying $\mathbb{E}(Z) \leq L_{K} \frac {\rho(\Sigma)}{n}$, such that
\[2 \langle \varepsilon,\hat{t}_{\hat m} - t^{\ast} \rangle_n - \pen(\hat{m}) \leq a^{-1} \left \| \hat{t}_{\hat m} - t^{\ast} \right \|_n^{2} + Z.\]
\end{prop}
According to the proposition, we find that
\[\mathbb{E} \left[ \left \| t^{\ast} - \hat{t}_{\hat m} \right \|_n^{2} \right] \leq \mathbb{E} \left[ \left \| t^{\ast} - \hat{t}_{m} \right \|_n^{2} \right] + \pen(m) + a^{-1} \mathbb{E} \left[ \left \| \hat{t}_{\hat m} - t^{\ast} \right \|_n^{2} \right] + \mathbb{E}(Z)\]
and
\[\frac{a-1}{a} \mathbb{E} \left[ \left \| t^{\ast} - \hat{t}_{\hat m} \right \|_n^{2} \right] \leq \mathbb{E} \left[ \left \| t^{\ast} - \hat{t}_{m} \right \|_n^{2} \right] + \pen(m) + L_{K} \frac {\rho(\Sigma)}{n}.\]
Thus, 
\[\mathbb{E} \left[ \left \| t^{\ast} - \hat{t}_{\hat m} \right \|_n^{2} \right] \leq C_{K} \left( \mathbb{E} \left[ \left \| t^{\ast} - \hat{t}_{m} \right \|_n^{2} \right] + \frac {\rho(\Sigma)}{n}+ \pen(m) \right),\]
where $C_{K} = \max \left( \frac{a}{a-1},  \frac{a L_{K}}{a-1} \right)$ and the proof of Theorem~\ref{theo:main_theorem} is complete. 

\subsection{Proof of Proposition~\ref{prop:pen_Z}}
 
We first recall a well known inequality from Cirel'son, Ibragimov et Sudakov~\cite{cirel1976norms}.
\begin{theo} \label{Ineg:Ibragimov}
Let $F : (\mathbb{R}^{n}, \| \cdot  \|) \rightarrow \mathbb{R}$  be a $1$-Lipschitz function and $\eta$ a random vector in $\R^n$ such that  $\eta \sim \mathcal{N}_n(0,\sigma^{2} Id)$ for some $\sigma >0$. Then there exists a random variable $\xi$ following an exponential distribution of parameter $1$ such that
\[F(\eta) \leq \mathbb{E} \left[ F(\eta) \right] + \sigma \sqrt{2 \xi}.\]
\end{theo}
Note that the Lipschitz condition is expressed with respect to  the (non-normalized) euclidean norm  $\| \cdot  \|$ in $\R^n$. We derive the following lemma for the projection of Gaussian random vectors. 
\begin{lem}
Let $\Sigma$ be a $n\times n$ symmetric semidefinite matrix and $S$ a linear subspace of $\R^n$. Let $\varepsilon$ be a Gaussian random vector such that $\varepsilon \sim \mathcal{N}_n(0, \Sigma)$. Then there exists a random variable $\xi$ following an exponential distribution of parameter $1$ such that
\[\left \| \Proj_{S} (\varepsilon) \right \|_n \leq \mathbb{E} \left \| \Proj_{S} (\varepsilon) \right \|_n + \sqrt{ \frac {\rho(\Sigma)}{n}} \sqrt{2 \xi}.\] 
\label{chap4:concentration}
\end{lem}
\begin{proof}
 Let $\varepsilon \sim \mathcal{N}_n(0, \Sigma)$, then $\varepsilon$ satisfies $\varepsilon = \sqrt{\Sigma} \eta$ with $\eta \sim \mathcal{N}_n(0, Id)$.
Let  $S$ be a linear subspace of $\R^n$.  We then check that the function $\eta \rightarrow \left \| \Proj_{S} (\sqrt{\Sigma} \eta) \right \|_n$ is a Lipschitz function
\begin{eqnarray*}
\left \| \Proj_{S} (\sqrt{\Sigma} x) - \Proj_{S} (\sqrt{\Sigma} y) \right \|_n   &\leq & \left \| \sqrt{\Sigma} (x-y) \right \|_n \\
&\leq &   \rho ( \sqrt \Sigma )  \left \| x - y \right \|_n \\
&\leq &   \sqrt{\rho ( \Sigma )} \left \| x - y \right \|_n =  \sqrt{\frac{\rho ( \Sigma )}{n}} \left \| x - y \right \| .
\end{eqnarray*}
By applying Theorem~\ref{Ineg:Ibragimov} to the function $\eta \rightarrow \left \| \Proj_{S} (\sqrt{\Sigma} \eta) \right \|_n$, we find that
\begin{equation*}
\left \| \Proj_{S} (\sqrt{\Sigma} \eta) \right \|_n \leq \mathbb{E} \left \| \Proj_{S} (\sqrt{\Sigma} \eta) \right \|_n + \sqrt{\frac{\rho ( \Sigma )}{n}} \sqrt{2 \xi}.
\end{equation*}
\end{proof}

We  are now in position to prove Proposition~\ref{prop:pen_Z}. 
Let $\bar{S}_{m}$  be the linear space spanned by  $S_{m}$ and $t^{\ast}$.   
By applying the inequality  $2 \langle x, y \rangle_n \leq a \| x \|_n^{2} + \| y \|_n^{2}/a$ for $a > 1$, we find that
\begin{eqnarray*}
2 \langle \varepsilon, \hat{t}_{\hat m} - t^{\ast} \rangle_n - \pen(\hat{m}) &= &2 \langle \Proj_{\bar{S}_{\hat{m}}} (\varepsilon), \hat{t}_{\hat m} - t^{\ast} \rangle_n - \pen(\hat{m}) \\
&\leq& a \left \| \Proj_{\bar{S}_{\hat{m}}} (\varepsilon) \right \|_n^{2} + a^{-1} \left \| \hat{t}_{\hat m} - t^{\ast} \right \|_n^{2} - \pen(\hat{m}) \\
&\leq& Z + a^{-1} \left \| \hat{t}_{\hat m} - t^{\ast} \right \|_n^{2},
\label{chap4:serie_ineg_Z}
\end{eqnarray*}
where 
\[\ Z =  a \left \| \Proj_{\bar{S}_{\hat{m}}} (\varepsilon) \right \|_n^{2} - \pen(\hat{m}). \]
Now,  we can write that
\begin{eqnarray*}
 \mathbb{E} (Z)=
 \mathbb{E} \left[ a \left \| \Proj_{\bar{S}_{\hat{m}}} (\varepsilon) \right \|_n^{2} - \pen(\hat{m}) \right] &\leq& a \mathbb{E} \left[ \max_{m \in \mathcal{M}} \left( \left \| \Proj_{\bar{S}_{m}} (\varepsilon) \right \|_n^{2} - \frac{1}{a} \pen(m) \right) \right] \\
&\leq&  a \sum_{m \in \mathcal{M}} \mathbb{E} \left[ \left( \left \| \Proj_{\bar{S}_{m}} (\varepsilon) \right \|_n^{2} - \frac{1}{a} \pen(m) \right)_{+} \right]. 
 \end{eqnarray*}

Let $m\in \mathcal M$. We start from the elementary inequality
\begin{eqnarray} \label{eq:CS}
\mathbb{E}  \left  \| \Proj_{\bar{S}_{m}} (\varepsilon) \right \|_n &\leq &  \left (\mathbb{E} \left \| \Proj_{\bar{S}_{m}} (\varepsilon) \right \|_n^{2} \right)^{1/2}  .
 \end{eqnarray}
By permuting the matrices  inside the trace operator, we can show that the quantity on the right side in \eqref{eq:CS} is exactly  equal to  $\sqrt{\frac{1}{n} \tr \left( \Proj_{\bar{S}_{m}} \Sigma   \right)}$. However $\bar{S}_m$ is unknown because it depends on $ t^{\ast}$ and thus we can not directly define the penalty in function of  $\tr \left( \Proj_{\bar{S}_{m}} \Sigma   \right)  $. 
We then use the decomposition 
\[ \Proj_{ \bar S_{m}}   =\Proj_{ S_{m}}    \oplus^{\perp}  \Proj_{ V_{m}}, \]
where  $V_m$ is the orthogonal to $S_m$ in $ \bar S_{m}$. Note that the dimension of $V_m$  is (at most) one.  By Pythagoras theorem
$
    \left \|  \Proj_{\bar S_m} (\varepsilon)  \right \|_n ^2  =  \left \|  \Proj_{ S_m}  (\varepsilon) \right \|_n ^2   +    \left \|  \Proj_{V_m} (\varepsilon)  \right \|_n ^2  
$. Now
$$
\E  \left \|  \Proj_{S_m} (\varepsilon)  \right \|_n^2=  \frac 1 n \tr \E \left ( \varepsilon^{t} \Proj_{S_m} \varepsilon \right )=  \frac 1 n \tr \E \left ( \varepsilon \varepsilon^{t} \Proj_{S_m}  \right )=
\frac 1 n \tr \left ( \Sigma \Proj_{S_m}  \right )= \frac 1 n \tr \left (  \Proj_{S_m}  \Sigma \right ) \, ,
$$
and 
$$
\E  \left \|  \Proj_{V_m} (\varepsilon)  \right \|_n^2= \frac 1 n \tr \left (  \Proj_{V_m}  \Sigma \right ) \leq \frac{\rho ( \Sigma)}{n} \, .
$$
 Finally
\begin{equation} \label{eq:StildeS}
  \E  \left \|  \Proj_{\bar S_m} (\varepsilon)  \right \|_n^2   \leq   \frac1n \tr \left( \Proj_{S_{m}}  \Sigma   \right)+ \frac {\rho(\Sigma)}{n}.
 \end{equation}

 According to Lemma~\ref{chap4:concentration} and using the inequalities  \eqref{eq:CS} and  \eqref{eq:StildeS}, there exists a random variable $\xi_m$ following an exponential distribution of parameter $1$ such that
\begin{equation*}
\left \| \Proj_{\bar{S}_{m}} (\varepsilon) \right \|_n \leq  \sqrt{  \frac 1 n  \tr \left( \Proj_{ S_{m}} \Sigma   \right) +  \frac {\rho(\Sigma)}{n} }+ \sqrt{\frac {\rho(\Sigma)}{n}} \sqrt{2 \xi_{m}}.
\end{equation*}
Thus, the random variable $Z$ satisfies
\begin{eqnarray*} 
\mathbb{E}(Z)  &\leq&  a \sum_{m \in \mathcal{M}} \mathbb{E} \left[ \left( \left \| \Proj_{\bar{S}_{m}} (\varepsilon) \right \|_n^{2} - \frac{1}{a} \pen(m) \right)_{+} \right] \\
&\leq &  a \sum_{m \in \mathcal{M}} \mathbb{E} \left[ \left( \left(  \sqrt{ \frac1n \tr \left( \Proj_{ S_{m}} \Sigma   \right) +  \frac {\rho(\Sigma)}{n} }+ \sqrt{\frac {\rho(\Sigma)}{n}} \sqrt{2 \xi_{m}}  \right)^{2} - \frac{1}{a} \pen(m) \right)_{+} \right].
 \end{eqnarray*}
We assume as in \eqref{eq:pen0} that 
\[\pen(m) \geq  \frac Kn  \left(  \sqrt{ \tr \left( \Proj_{ S_{m}} \Sigma   \right) +  \rho(\Sigma) }+ \sqrt{\rho(\Sigma)} \sqrt{2 \log \left( \frac{1}{\pi_{m}} \right) }  \right)^{2} . \]
Then,
\begin{multline*}  
\mathbb{E}(Z) \leq \frac an \sum_{m \in \mathcal{M}} \mathbb{E} \Bigg[ \Bigg( \left( \sqrt{    \tr \left( \Proj_{ S_{m}} \Sigma   \right) +  \rho(\Sigma) } + \sqrt{\rho(\Sigma)} \sqrt{2 \xi_{m}} \right)^{2} \\
- \frac{K}{a} \left( \sqrt{  \tr \left( \Proj_{ S_{m}} \Sigma   \right) +  \rho(\Sigma)  } + \sqrt{ \rho(\Sigma) } \sqrt{2 \log \left( \frac{1}{\pi_{m}} \right)} \right)^{2} \Bigg)_{+} \Bigg].
\end{multline*}
Using the inequality $(x + y)^{2} \leq (1 + \alpha) x^{2} + (1 + \alpha^{-1}) y^{2}$, and taking  $\alpha = \frac{K-a}{a}$ for $K >a  >1$, we find that 
\begin{multline*} 
\left( \sqrt{  \tr \left( \Proj_{ S_{m}} \Sigma   \right) +  \rho(\Sigma)} + \sqrt{\rho(\Sigma)} \sqrt{2 \xi_{m}} \right)^{2} \\
 \leq  \left( \sqrt{   \tr \left( \Proj_{ S_{m}} \Sigma   \right) +  \rho(\Sigma)  } + \sqrt{\rho(\Sigma)} \sqrt{2 \log \left( \frac{1}{\pi_{m}} \right)} + \sqrt{\rho(\Sigma)} \sqrt{2 \left( \xi_{m} - \log \left( \frac{1}{\pi_{m}} \right) \right)_{+}} \right)^{2} \\
 \leq \frac{K}{a} \left( \sqrt{   \tr \left( \Proj_{ S_{m}} \Sigma   \right) +  \rho(\Sigma)} + \sqrt{\rho(\Sigma)} \sqrt{2 \log \left( \frac{1}{\pi_{m}} \right)} \right)^{2} + \frac{2K \rho(\Sigma)}{K-a} \left( \xi_{m} - \log \left( \frac{1}{\pi_{m}} \right) \right)_{+}.
\end{multline*}
Next,
\begin{multline*}
\mathbb{E} \Bigg[ \Bigg( \left( \sqrt{ \tr \left( \Proj_{ S_{m}} \Sigma   \right) +  \rho(\Sigma)} + \sqrt{\rho(\Sigma)} \sqrt{2 \xi_{m}} \right)^{2} \\
 - \frac{K}{a} \left( \sqrt{ \tr \left( \Proj_{ S_{m}} \Sigma   \right) +  \rho(\Sigma)} + \sqrt{\rho(\Sigma)} \sqrt{2 \log \left( \frac{1}{\pi_{m}} \right)} \right)^{2} \Bigg)_{+} \Bigg] \\
\leq \mathbb{E} \left[ \frac{2K \rho(\Sigma)}{K-a} \left( \xi_{m} - \log \left( \frac{1}{\pi_{m}} \right) \right)_{+} \right] 
\leq \frac{2K \rho(\Sigma)}{K-a} \pi_{m},
\end{multline*}
because $\mathbb{E} \left[ \left( \xi_{m} - \log \left( \frac{1}{\pi_{m}} \right) \right)_{+} \right] = \exp(- \log(\frac{1}{\pi_{m}})) = \pi_{m}$.
Since $\sum_{m \in \mathcal{M}} \pi_{m} = 1$, we finally obtain that
\begin{equation*} 
\mathbb{E}(Z) \leq   a \sum_{m \in \mathcal{M}} \frac{2K}{K-a} \pi_{m}  \frac {\rho(\Sigma)}{n}
 = \frac{2aK  }{K-a} \frac {\rho(\Sigma)}{n} \, .
 \end{equation*}
For any $K >1$, take $a = \frac{K+1}{2}$. Then $K >a  >1$ is satisfied and the proof of Proposition~\ref{prop:pen_Z} is complete with $L_{K} =  \frac{2K^2+2K}{K-1}$.
\subsection{Proof of Lemma~\ref{mainLem}}

\setcounter{MaxMatrixCols}{20}

For any $m \in \{1, \ldots, n\}$ and any $j \in  \{1, \ldots, m\}$, we define the discrete interval
$$
I_j= \left \{ i \in \{1, \ldots, n\} : \frac i n \in \left [ \frac{(j-1)}{m}, \frac j m \right [ \right \} \, ,
$$
and we denote by $\ell(j)$ the length of $I_j$: $\ell(j)= \text{Card}(I_j)$. Note that, for all $j$, $[n/m]\leq \ell_j \leq [n/m]+1$. The linear space
$S_m$   induced by the family of piecewise polynomials of degree at most $r$  on the regular partition of size $m$ of the interval $[0,1]$ is the space generated by the $(r+1)m$ columns of the design
\[  X=
\begin{pmatrix} 
1 & 1  & \dots & 1         & 0 & 0 &  \dots & 0 & 0 & \dots  & 0  & 0 &  \dots & 0\\
 1 & 2 &  \dots & 2^r     & 0 & 0 &  \dots &  0 & 0 & \dots  & 0 & 0 & \dots & 0 \\
\vdots & \vdots & \vdots & \vdots &  \vdots & \vdots & \vdots  & \vdots & \vdots & \vdots  & \vdots  & \vdots & \vdots & \vdots \\
 1 & \ell_1 &  \dots & \ell_1^r  & 0 & 0 &  \dots & 0 & 0 & \dots & 0 & 0 & \dots & 0 \\
 0 & 0 &  \dots & 0 & 1 & 1  & \dots & 1 & 0 & \dots   & 0 & 0 & \dots & 0\\
 0 & 0 & \dots  & 0 & 1 & 2 &  \dots & 2^r & 0 & \dots  & 0  & 0  & \dots & 0 \\
 \vdots & \vdots & \vdots & \vdots & \vdots & \vdots & \vdots  & \vdots & \vdots & \vdots  & \vdots  & \vdots & \vdots  & \vdots \\
 0 & 0 &  \dots & 0 & 1 & \ell_2 &  \dots & \ell_2^r & 0 & \dots & 0  & 0 & \dots & 0 \\
 \vdots & \vdots & \vdots & \vdots & \vdots & \vdots & \vdots  & \vdots & \vdots & \vdots  & \vdots  & \vdots & \vdots & \vdots \\
  \vdots & \vdots & \vdots & \vdots & \vdots & \vdots & \vdots  & \vdots & \vdots & \vdots  & \vdots  & \vdots  & \vdots & \vdots \\
 0 & 0 &  \dots & 0 & 0 & 0 &  \dots & 0 & 0 & \dots & 1 & 1 & \dots  & 1 \\
 0 & 0 &  \dots & 0 & 0 & 0 &  \dots & 0 & 0 & \dots & 1  & 2 & \dots  & 2^r\\
 \vdots & \vdots & \vdots & \vdots &  \vdots & \vdots & \vdots  & \vdots & \vdots & \vdots  & \vdots  & \vdots & \vdots  & \vdots \\
 0 & 0 &  \dots & 0 & 0 & 0 &  \dots & 0 & 0 & \dots & 1 & \ell_m & \dots & \ell_m^r
\end{pmatrix}. 
\]

Let $c_k$ be the $k$-th column of the matrix $X$. Note that these columns are not all orthogonal, but they are linearly independent. 

For $k \in \{1, \ldots, m\}$, let $V_k$ be the linear subspace of ${\mathbb R}^n$  generated by 
the $c_j$'s for $j \in \{(k-1)(r+1) +1, \ldots, k(r+1) \}$.  Note that the subspaces $V_k$ are orthogonal subspaces, so that 
$$
\left \|\Proj_{S_m}(\varepsilon) \right \|_n^2= \sum_{k=1}^m \left \|\Proj_{V_k}(\varepsilon) \right \|_n^2.
$$
We shall  prove that there exists a constant $C>0$ such that, for any $k \in \{1, \ldots, m\}$, 
\begin{equation} \label{projV}
n  {\mathbb E} \left (\left \|\Proj_{V_k}(\varepsilon) \right \|_n^2  \right )  \leq C  \frac{n^{1-\gamma}}{m^{1-\gamma}} \, . 
\end{equation}
If \eqref{projV} is true then the proof of Lemma \ref{mainLem} is easy to complete. Indeed 
$$
 \tr \left( \Proj_{ S_{m}} \Sigma   \right) =  n {\mathbb E} \left (  \left \|\Proj_{S_m}(\varepsilon) \right \|_n^2 \right ) 
= \sum_{k=1}^m n  {\mathbb E} \left (\left \|\Proj_{V_k}(\varepsilon) \right \|_n^2 \right ) \leq Cm^\gamma n^{1- \gamma} \, .
$$

It remains to prove \eqref{projV}. In fact, it suffices to prove \eqref{projV} for $V_1$, the argument being unchanged for the other 
$V_k$'s. Let $e_k=  c_k/ \sqrt{c_k^t c_k}$, so that $n \|e_k\|_2^2=1$, and let $X_1$ the 
$n \times (r+1)$ matrix composed of the $(r+1)$ columns $e_1, \ldots e_{r+1}$. We can write 
$$
  \Proj_{V_1}(\varepsilon)= \alpha_1 e_1 + \cdots + \alpha_{r+1} e_{r+1},
$$
where
$$
   (\alpha_1, \ldots, \alpha_{r+1})^t = (X_1^t X_1)^{-1} X_1^t \varepsilon \, .
$$
Clearly
\begin{equation} \label{rough}
   \sqrt{\sum_{k=1}^{r+1} \alpha_k^2} \leq \rho \left ( (X_1^t X_1)^{-1}  \right )  \sqrt{\sum_{k=1}^{r+1} (e_k^t \varepsilon)^2} \, , 
\end{equation}
where $ \rho \left ( (X_1^t X_1)^{-1}  \right )$ is the spectral radius of $(X_1^t X_1)^{-1} $. Since 
$$
 n  \left \|\Proj_{V_1}(\varepsilon) \right \|_n^2  \leq (r+1)^2 \sum_{k=1}^{r+1} \alpha_k^2 \, ,
$$
we infer from \eqref{rough} that 
\begin{equation}\label{close}
n  {\mathbb E} \left (\left \|\Proj_{V_1}(\varepsilon) \right \|_n^2 \right ) \leq \left ( (r+1) \rho \left ( (X_1^t X_1)^{-1}  \right )\right )^2  \, \sum_{k=1}^{r+1} {\mathbb E}\left ( (e_k^t \varepsilon)^2 \right ) \, .
\end{equation}
Before going further, we need to check that $ \rho \left ( (X_1^t X_1)^{-1}  \right )$ is uniformly bounded: indeed this quantity depends on the length $\ell_1$, which can be as large as $n$. This is true, because $ X_1^t X_1 $ tends to $A$ as $\ell_1 \rightarrow \infty$, where $A$ is an invertible $(r+1)  \times (r+1)$ matrix  (in fact one can  check that $A_{i,j}= \sqrt{(2j+1)(2i+1)}/(j+i+1)$).  It follows that, as $\ell_1$ varies,  $ \rho \left ( (X_1^t X_1)^{-1}  \right )$ is a sequence of positive numbers converging to 
$\rho(A^{-1})$: it is therefore uniformly bounded. It follows from \eqref{close}  that there exists $K>0$ such that 
$$
n  {\mathbb E} \left (\left \|\Proj_{V_1}(\varepsilon) \right \|_n^2 \right ) \leq 
K \sum_{k=1}^{r+1} {\mathbb E}\left ( (e_k^t \varepsilon)^2 \right ) \, .
$$
Hence \eqref{projV} will be proved  for $V_1$ if there exists $C_1>0$ such that, for any $k \in \{1, \ldots, r+1\}$, 
\begin{equation}\label{coeffproj}
 {\mathbb E}\left ( (e_k^t \varepsilon)^2 \right )= {\mathbb E}\left ( \left (\frac{c_k^t \varepsilon}{\sqrt {c_k^t c_k}} \right )^2 \right )  \leq C_1  \frac{n^{1-\gamma}}{m^{1-\gamma}} \, . 
\end{equation}

It remains to prove \eqref{coeffproj}.
Let then $k \in \{1, \ldots, r+1\}$. By stationarity, 
$$
{\mathbb E}\left ( \left (c_k^t \varepsilon \right )^2 \right ) = \sum_{i=1}^{\ell_1}
\sum_{j=1}^{\ell_1} i^k j^k \gamma_\varepsilon (j-i)  \leq \gamma_\varepsilon(0) \sum_{i=1}^{\ell_1} i^{2k} 
+2 \sum_{j=1}^{\ell_1} |\gamma_\varepsilon(j)| \sum_{i=1}^{\ell_1-j} i^k (i+j)^k \, .
$$
Now, by Cauchy-Schwarz, 
$$
\sum_{i=1}^{\ell_1-j} i^k (i+j)^k \leq \sum_{i=1}^{\ell_1} i^{2k}  = c_k^t c_k \, .
$$
Combining the two last inequalities, we get 
\begin{equation}\label{finalLem}
 {\mathbb E}\left ( \left (\frac{c_k^t \varepsilon}{\sqrt {c_k^t c_k}} \right )^2 \right ) \leq \gamma_{\varepsilon}(0) + 2\sum_{j=1}^{\ell_1}  |\gamma_\varepsilon(j)| \, .
\end{equation}
Now, recall that \eqref{lrd} holds, that is 
$
     |\gamma_\varepsilon (k)| \leq \kappa (k+1)^{-\gamma} 
$ for some $\kappa >0$ and $\gamma \in (0,1)$. From \eqref{finalLem}, we easily infer that there exists $C_2>0$ such that 
$$
 {\mathbb E}\left ( \left (\frac{c_k^t \varepsilon}{\sqrt {c_k^t c_k}} \right )^2 \right ) \leq C_2 \ell_1^{1-\gamma} \, .
$$
Since $[n/m]\leq \ell_1 \leq [n/m]+1$, \eqref{coeffproj} easily follows. This completes the proof of Lemma \ref{mainLem}.


\subsection{Proof of Lemma~\ref{newLem}}

We keep the notations of the proof of Lemma \ref{mainLem}. Recall that the case of regular regressograms corresponds to the degree $r=0$. In that case, the design matrix $X$ of the proof of Lemma \ref{mainLem} contains only the $m$ orthogonal columns filled with 0 and 1, and the linear space $S_m$ has dimension $m$. Denote by $c_1, \ldots, c_m$ the $m$ columns of the design $X$.

We can write the exact expression of $ \Proj_{S_m}(\varepsilon)$
$$
  \Proj_{S_m}(\varepsilon)= \bar \varepsilon_1 c_1 + \bar \varepsilon_2 c_2 + \cdots + \bar \varepsilon_m c_m \,, \quad \text{with}  \quad \bar \varepsilon_k= \frac{1}{\ell_k} \sum_{i \in I_k} \varepsilon_i \, .
$$
Consequenly
$$
 n \left \|\Proj_{S_m}(\varepsilon) \right \|_n^2 = \ell_1  \bar \varepsilon_1^2 + \ell_2  \bar \varepsilon_2^2 + \cdots + \ell_m  \bar \varepsilon_m^2 \, .
$$
Now, it follows from \eqref{vsd} that ${\mathbb E} (\bar \varepsilon_i^2) \leq \kappa \ell_i^{-\gamma}$. Hence 
$$
 \tr \left( \Proj_{ S_{m}} \Sigma   \right) =  n {\mathbb E} \left (  \left \|\Proj_{S_m}(\varepsilon) \right \|_n^2 \right ) \leq \kappa \sum_{k=1}^m \ell_k^{1-\gamma} \, .
$$
Since, for all $j$, $[n/m]\leq \ell_j \leq [n/m]+1$, we infer that there exists a positive constant $C$ depending only on $\kappa$ and $\gamma$ such that 
$$
 \tr \left( \Proj_{ S_{m}} \Sigma   \right)  \leq C m^\gamma n^{1-\gamma} \, .
$$
This concludes the proof of Lemma \ref{newLem}.

\bibliography{select_gauss_arxiv_bib}
\bibliographystyle{amsalpha}

\end{document}